\documentclass[12pt,a4paper]{report}

\usepackage[utf8]{inputenc}
\usepackage[T2A]{fontenc}

\usepackage{amsmath,amssymb,amsthm,amsfonts}
\usepackage{geometry}
\usepackage{enumitem}
\setlist[enumerate]{itemsep=0.3ex, topsep=0.3ex, label={\rm(\arabic*)}}
\setlist[itemize]{itemsep=0.3ex, topsep=0.3ex, leftmargin=4ex}

\usepackage{indentfirst}
\usepackage{graphicx}
\graphicspath{{figs/}}

\usepackage[all]{xy}

\usepackage{xcolor}
\usepackage{hyperref}
\hypersetup{hidelinks=true, hypertexnames=false}
\hypersetup{
   hypertexnames=false,
   colorlinks,
   linkcolor={red!50!black},
   citecolor={blue!50!black},
   urlcolor={blue!80!black}
}

\usepackage[backend=biber,style=numeric-comp,sorting=none]{biblatex}
\DeclareMathOperator{\cat}{cat}
\DeclareMathOperator{\Pcat}{P\text{-}cat}

\DeclareMathOperator{\Int}{Int}

\newcommand\R{\mathbb{R}}
\newcommand\Z{\mathbb{Z}}
\newcommand\Q{\mathbb{Q}}

\usepackage[english]{babel}

\newtheorem{theorem}{Theorem}[chapter]

\newtheorem{proposition}[theorem]{Proposition}
\newtheorem{corollary}[theorem]{Corollary}
\theoremstyle{definition}
\newtheorem{definition}[theorem]{Definition}

\newtheorem{remark}[theorem]{Remark}

\begin{document}

\begin{titlepage}
\begin{center}
\textsc{Academy of Sciences of Ukraine}\\
\textsc{Institute of Mathematics}
\vfill
\hfill\textit{Manuscript copyright}
\vspace{1cm}

{\large \textbf{Bondar Olga Petrivna}}
\vspace{1cm}

{\Large \textbf{$P$-categories and functions with\\ degenerate singular submanifolds}}
\vspace{0.5cm}

01.01.01 --- Mathematical analysis
\vspace{1cm}

\textsc{Dissertation}\\
submitted for the degree of\\
Candidate of Physical and Mathematical Sciences
\vfill

\hfill\begin{minipage}{0.5\textwidth}
Scientific supervisor:\\
Doctor of Physical and\\
Mathematical Sciences, Professor\\
\textbf{V.\,V.\,Sharko}
\end{minipage}
\vfill

Kyiv --- 1993
\end{center}
\end{titlepage}

\tableofcontents
\newpage
\chapter*{Introduction}
\addcontentsline{toc}{chapter}{Introduction}

One of the central topics in the theory of functions on manifolds is the study of the interrelation between the topology of a manifold and the critical points of functions defined on it.
The first indication of a connection between the critical points of a function and the topology of its domain of definition was contained already in the Birkhoff minimax principle~\cite{birkhoff-1917}, which gave a lower bound on the number of saddle points of a function defined on a two-dimensional manifold in terms of the number of relative minima and the homologies of that manifold.
Inspired by Birkhoff's results, in 1926 there appeared the paper of M.\,Morse~\cite{morse-1925}, which laid the foundation of a qualitatively new theory, named after its author.
Morse showed~\cite{morse-1934} that the number of critical points of various indices of a differentiable function on a manifold is related to the homologies of that manifold; namely, he proved inequalities connecting the numbers of critical points with the ranks and torsion orders of the homology groups of the manifold.

Morse theory developed in various directions through the work of many mathematicians.
Since it relates the singularities of a function to the topology of the underlying space, its development is mainly oriented towards using information about one of the actors (the singularities of the function or the topology of the space on which the function is defined) to describe the other.
Thus R.\,Thom proved~\cite{thom-1949} the existence on a topological space of a cell-space structure whose cells correspond to the critical points of a function defined on it.
This allowed one to draw conclusions about the homotopy type of the space.
S.\,Smale~\cite{smale-1964} discovered the ``handle decomposition'' --- one handle for each critical point --- which made it possible to obtain a result about the differentiable-homotopy type.
R.\,Bott~\cite{bott-1982} related the structure of a manifold endowed with a differentiable function to its non-degenerate critical submanifolds.

In all of these studies an important role is played by the estimate of the number of singularities of a function on a manifold, primarily by how small the number of singularities can be.
Thus, for instance, the theorem of S.\,Smale on the exactness of the Morse inequalities implies the Poincaré conjecture in dimensions greater than five, as well as the $h$-cobordism theorem.
The existence on a smooth compact manifold without boundary of a function with two critical points implies that the manifold is homeomorphic to a sphere (the Reeb theorem~\cite{reeb-1947} and its generalisation by D.\,Milnor~\cite{milnor-difftop} to degenerate critical points).

For estimating the minimal number of non-degenerate critical points one used the well-known Morse inequalities~\cite{milnor-morse-1965}:
\[
   B_k(M) - B_{k-1}(M) + \ldots \pm B_0(M)
   \;\leq\; C_k - C_{k-1} + \ldots \pm C_0,
\]
where $C_k$ is the number of critical points of index $k$ on a compact manifold $M$, $B_k = \dim H_k(M, A)$, and $A$ is one of the groups $\Z$, $\Q$, $\Z_p$ ($p \neq 2$, $p$ prime).

For estimating the minimal number of degenerate critical points no such estimate existed.
For this reason L.\,A.\,Lyusternik and L.\,G.\,Shnirelman~\cite{lyusternik-shnirelman-1935} posed the following general problem:
\begin{quote}
Is it possible to indicate a topological characteristic of a space on which a function (or functional) is defined, which would characterise the smallest possible number of critical points of such a function?
\end{quote}

The answer was given in~\cite{lyusternik-shnirelman-1935} by a topological invariant named after its authors --- the Lyusternik--Shnirelman category.
It was defined as the minimal cardinality of a covering of the manifold by closed subsets that are contractible within the manifold.

Another topological invariant of a manifold also answers this question --- its \emph{length}: the largest number of cycles of the manifold whose intersection is a cycle not homologous to zero.
The notion of length of a manifold was introduced by S.\,V.\,Frolov and L.\,E.\,Elsgolts~\cite{frolov-elsgolts-1939}.
They also proved that the Lyusternik--Shnirelman category of a manifold without boundary is at least one greater than its length.
Consequently, the number $k$ of critical points of a differentiable function on a smooth compact connected manifold $M$ without boundary satisfies the inequalities
\[
   k \;\geq\; \cat M \;\geq\; \mathrm{long}\, M + 1.
\]

The natural question arises: can these inequalities be generalised to functions with degenerate submanifolds?
To answer this question, the present work introduces the following notions:
\begin{itemize}\setlength{\itemsep}{0pt}
\item the definition of a $P$-function (Definition~\ref{def:2.1}), generalising the notion of a function with isolated critical points;
\item the notion of $P$-category (Definition~\ref{def:1.1}), generalising the notion of Lyusternik--Shnirelman category;
\item the notion of $p$-length of a manifold (Definition~\ref{def:1.3}), generalising the notion of length of a manifold.
\end{itemize}

Then the number $k$ of critical submanifolds of a $P$-function on the corresponding manifold $M$ (Theorem~\ref{thm:1.5}, Theorem~\ref{thm:2.12}) satisfies the inequalities
\[
   k \;\geq\; \Pcat M \;\geq\; \mathrm{long}^{\,p} M + 1.
\]

Using these inequalities the work estimates, besides the minimal number of critical submanifolds of $P$-functions on manifolds, also the $P$-category and the $p$-length of several manifolds, and constructs $P$-functions on two-dimensional and three-dimensional manifolds, as well as on some manifolds of higher dimension.
We now describe the results obtained in more detail.

In Chapter~I, ``$P$-category of a topological space'', we define the $P$-category of a topological space, introduce the notion of $p$-length of a manifold, in terms of which a lower bound on the $P$-category is obtained, and give examples of computation of the $p$-length and $P$-category of manifolds.

\begin{definition}[1.1]
Let $A$, $B$, $P$ be closed subsets of a Hausdorff topological space, with $P \subset B$ and $A \subset B$. The \emph{$P$-category} $\Pcat_B A$ of the set $A$ relative to the set $B$ is the minimal number $k$ of closed subsets $A_1, \ldots, A_k$ in $B$ that satisfy the following properties:
\begin{enumerate}\setlength{\itemsep}{0pt}
\item $A$ is their union
   \[
      A = \bigcup A_i;
   \]
\item each subset $A_i$, $i=1,\ldots,k$, can be contracted within $B$ to $P$.
That is, for each $i$ there exists a homotopy
      \[
         F_i\colon A_i \times I \to B \;\text{such that}
      \]
      \begin{itemize}\setlength{\itemsep}{0pt}
\item[a)] $F_i(x,0) = x$, $x \in A_i$;
\item[b)] $F_i(x,t) \in B$, $x \in A_i$, $0 < t < 1$;
\item[c)] $F_i(x,1)$ is homeomorphic to $P$.
      \end{itemize}
\end{enumerate}
If no such number $k$ exists we set $\Pcat_B A = \infty$.
The subsets $A_1, \ldots, A_k$ are called \emph{categorical}.
\end{definition}

Note that if in the definition of $P$-category one takes $P$ to be a point, then this definition coincides with the definition of the Lyusternik--Shnirelman category. The $P$-category of a manifold can be estimated in terms of its length (a corollary of Property~4):
\[
   \Pcat M \;\geq\; \frac{\mathrm{long}\, M + 1}{\mathrm{long}\, P + 1}.
\]

In particular, when $P$ is a point the well-known estimate of the Lyusternik--Shnirelman category of a manifold $M$
\[
   \cat M \;\geq\; \mathrm{long}\, M + 1
\]
coincides with the estimate above.

In many cases the estimate of the $P$-category of a manifold in terms of its length is not sharp; that is, for many well-known manifolds the $P$-category is strictly greater than this lower bound expressed in terms of length.
With the aim of sharpening the lower bound for $P$-category, the work introduces the notion of $p$-length of a manifold.

\begin{definition}[1.3]
Let $n$ be the dimension of the manifold $M$ and let $p$ be a non-negative integer.
Let $H^*(M; A)$ be the cohomology ring of $M$, where $A = \Z$ or $\Z_p$ when $M$ is orientable, and $A = \Z_2$ when $M$ is non-orientable.
Consider all integers $q$ such that there exist elements $a_1, \ldots, a_q$ in $H^*(M; A)$ satisfying
\begin{enumerate}\setlength{\itemsep}{0pt}
\item $\dim a_i > 0$, $i = 1, \ldots, q$;
\item the product $a_1 \cdots a_q$ (with respect to the multiplication in the cohomology ring $H^*(M; A)$) is non-zero;
\item $\dim (a_1 \cdots a_q) \leq n - p$.
\end{enumerate}
The \emph{cohomological $p$-length} of the manifold $M$ is the largest such number $q$.
\end{definition}

Note that if $p = 0$, the definition of $p$-length coincides with the classical notion of length.

It turns out to be possible to estimate the $P$-category of a manifold in terms of its $p$-length.
The following theorem is the main result of Chapter~I.

\begin{theorem}[1.5]
If $M$ is a manifold without boundary, then its $P$-category is at least one more than the $p$-length of $M$:
\[
   \Pcat M \;\geq\; \mathrm{long}^{\,p} M + 1.
\]
If $M$ is a manifold with boundary, then its $P$-category is at least the $p$-length of $M$:
\[
   \Pcat M \;\geq\; \mathrm{long}^{\,p} M.
\]
\end{theorem}

When $P$ is a point, the well-known theorem giving an estimate of the Lyusternik--Shnirelman category of a manifold without boundary (see, e.g.,~\cite{milnor-morse-1965})
\[
   \cat M \;\geq\; \mathrm{long}\, M + 1
\]
is a corollary of Theorem~\ref{thm:1.5}.
Using the obtained estimate of the $P$-category, the work gives examples of computation of the $P$-category and $p$-length of manifolds.
Interesting examples are obtained as corollaries of the following theorem.

\begin{theorem}[1.7]
Let
\[
   M = S^{k_1} \times S^{k_2} \times \ldots \times S^{k_r}
\]
be a product of $k_i$-dimensional spheres, $i = 1, \ldots, r$, and let $p$ be a non-negative integer.
Then
\[
   \mathrm{long}^{\,p} M = r - \min l(p),
\]
where the minimum is taken over all subsets
\[
   \{(k_{i_1}, k_{i_2}, \ldots, k_{i_l}) \subset (k_1, k_2, \ldots, k_r) : \textstyle\sum k_{i_j} \geq p\}.
\]
\end{theorem}

\begin{corollary}[1.8]
Let
\[
   M = S^{k_1} \times \ldots \times S^{k_r}
\]
be a product of $k_i$-dimensional spheres, $i = 1, \ldots, r$.
Then the (co)homological length of $M$ equals the number $r$ of factors:
\[
   \mathrm{long}\, M = r.
\]
\end{corollary}

\begin{proposition}[1.9]
The $p$-length of the $n$-dimensional torus is $p$ less than its dimension:
\[
   \mathrm{long}^{\,p} T^n = n - p.
\]
\end{proposition}

On two-dimensional manifolds it makes sense to consider only the \emph{round category} (Definition~\ref{def:1.2}).
We emphasise that all two-dimensional compact connected manifolds are classified in this work according to their $P$-category.

\begin{proposition}[1.11]
The round category of a two-dimensional compact connected manifold equals:
\begin{enumerate}\setlength{\itemsep}{0pt}
\item one, if the manifold is homeomorphic to $S^1 \times I$ or to the Möbius band;
\item two, on the remaining manifolds with boundary;
\item two, if the manifold is homeomorphic to the sphere $S^2$, the projective plane $\R P^2$, the torus $T^2$ or the Klein bottle;
\item three, on the remaining manifolds without boundary.
\end{enumerate}
\end{proposition}

Three-dimensional manifolds are only partially classified by round category, but the following proposition gives an upper bound for the round category.

\begin{proposition}[1.12]
The round category of a three-dimensional compact connected manifold without boundary
\begin{enumerate}\setlength{\itemsep}{0pt}
\item equals two if the manifold is homeomorphic to one of the following:
   \begin{itemize}\setlength{\itemsep}{0pt}
\item[a)] the sphere $S^3$;
\item[b)] $S^1 \times S^2$;
\item[c)] the projective space $\R P^3$;
\item[d)] a lens space obtainable as the quotient of $S^3$ by a differentiable action of the group $\Z_p$;
   \end{itemize}
\item equals three on the torus $T^3$;
\item does not exceed four on all other such manifolds.
\end{enumerate}
\end{proposition}

The following two theorems turn out to be the source for computing the $P$-category of manifolds of special type.

\begin{theorem}[1.13]
Let
\[
   M = S^{k_1} \times S^{k_2} \times \ldots \times S^{k_r}
\]
be a product of $k_i$-dimensional spheres, $i = 1, \ldots, r$.
Then the Lyusternik--Shnirelman category of $M$ is one more than the number $r$ of factors:
\[
   \cat M = r + 1.
\]
\end{theorem}

\begin{theorem}[1.14]
Let
\[
   M = S^{k_1} \times S^{k_2} \times \ldots \times S^{k_r}
\]
be a product of $k_i$-dimensional spheres, $k_i \geq 1$, $i = 1, \ldots, r$, and let
\[
   P = S^{k_{i_1}} \times S^{k_{i_2}} \times \ldots \times S^{k_{i_l}}
\]
be a submanifold of $M$, where $(k_{i_1}, k_{i_2}, \ldots, k_{i_l}) \subset (k_1, k_2, \ldots, k_r)$.
Then the $P$-category of $M$ equals
\[
   \Pcat M = r - l + 1.
\]
\end{theorem}

\begin{corollary}[1.15]
Let $P$ be the $k$-dimensional torus.
Then the $P$-category of the $n$-dimensional torus equals
\[
   T^k\text{-}\cat T^n = n - k + 1.
\]
\end{corollary}

In particular, the round category of the $n$-dimensional torus equals $n$.

In Chapter~II, ``Estimate of the number of critical submanifolds of a function on a manifold'', we give the definition of a $P$-function on a manifold, study a condition for their existence, give an estimate of the number of critical submanifolds of functions on a manifold in terms of its $P$-category, define exact $P$-functions and demonstrate their existence on certain manifolds.

\begin{definition}[2.1]
A differentiable function on a manifold whose set of critical points is a disjoint union of smooth submanifolds without boundary is called a \emph{$P$-function} if each of its critical submanifolds is homeomorphic to a fixed smooth manifold $P$ without boundary.
\end{definition}

The following theorem provides a sufficient condition for the existence of a $P$-function on a manifold.
It is the main result of \S2.1.

\begin{theorem}[2.5]
Let $M^m$ be a smooth compact manifold (with or without boundary), and let
\[
   M_1 \subset M_2 \subset \ldots \subset M_k = M^m
\]
be a filtration of $M^m$ by compact manifolds with boundary, satisfying:
\begin{enumerate}\setlength{\itemsep}{0pt}
\item[(a)] $M_i$ are manifolds with boundary, $i = 1, \ldots, k$;
\item[(b)] $M_i \subset \Int(M_{i+1})$, $i = 1, \ldots, k-1$;
\item[(c)] $\partial M_i \subset (M_i \setminus M_{i-1})$;
\item[(d)] for each $i = 1, \ldots, k$ the manifold $(M_i \setminus \Int(M_{i-1})) \cup \partial M_{i-1} \cup \partial M_i$
\end{enumerate}
admits a covering by three closed subsets as in Proposition~\ref{prop:2.4}.
Then on $M^m$ there exists a $P$-function whose number of critical submanifolds equals $k$.
\end{theorem}

The next theorem is the main result of the dissertation.

\begin{theorem}[2.12]
Let $M$ be a smooth compact connected manifold (with or without boundary), and let $f$ be a $P$-function on it which, in the case of a manifold with boundary, takes a constant maximal value on the boundary and has no critical points on the boundary.
Then the number of critical submanifolds of $f$ is at least the $P$-category of $M$.
\end{theorem}

Is the bound obtained sharp; that is, do there exist manifolds on which one can construct a $P$-function whose number of singularities equals the $P$-category?
An affirmative answer to this question is given by most of the examples considered below.

\begin{definition}[2.13]
A $P$-function on a manifold is called \emph{exact} if the number of its critical submanifolds equals the minimum, taken over all $P$-functions on the manifold, of the number of critical submanifolds.
\end{definition}

It follows that if the number of singularities of a $P$-function on a manifold coincides with the manifold's $P$-category, then the function is exact.
The converse is not known.

Almost all functions considered in Chapter~III will be exact.
We give some examples of the existence and construction of exact $P$-functions on manifolds.

\begin{definition}[2.14]
A $P$-function on a manifold is called \emph{round} if its critical submanifolds are homeomorphic to the circle $S^1$.
\end{definition}

In particular, the round Morse functions introduced by Thurston~\cite{thurston-1976}, and studied by Franks~\cite{franks-1980}, Asimov~\cite{asimov-1975}, Miyoshi~\cite{miyoshi-1983}, Morgan~\cite{morgan-1979}, A.\,T.\,Fomenko, H.\,Zieschang, S.\,V.\,Matveev, A.\,V.\,Brailov and V.\,V.\,Sharko \cite{matveev-fomenko-sharko-1988,sharko-1990,fomenko-sharko-1989}, form a subset of the round functions.

The following theorem is the source of constructions of exact $P$-functions on manifolds of special type.

\begin{theorem}[2.15]
Let
\[
   \{k_{i_1}, \ldots, k_{i_l}\}
\]
be a subset of
\[
   \{k_1, \ldots, k_n\}
\]
positive integers.
Let
\[
   P = S^{k_{i_1}} \times \ldots \times S^{k_{i_l}}
\]
be a product of $k_{i_j}$-dimensional spheres, $j = 1, \ldots, l$.
Then on the manifold
\[
   M = S^{k_1} \times \ldots \times S^{k_n}
\]
there exists an exact $P$-function whose number of singularities equals
\[
   n - l + 1.
\]
\end{theorem}

\begin{corollary}[2.16]
Let $P$ be the $k$-dimensional torus, $k \leq n$.
Then on the $n$-dimensional torus there exists an exact $P$-function whose number of singularities equals
\[
   n - k + 1.
\]
\end{corollary}

As an interesting consequence, on the $n$-dimensional torus there exists an exact round function with $n$ singularities.

The existence on a sphere of a function with two singularities is discussed in the following theorem.

\begin{theorem}[2.17]
Let $P = S^n$ be the $n$-dimensional sphere.
Then on the odd-dimensional sphere
\[
   S^{2n+1}, \quad n \geq 1,
\]
there exists an exact $P$-function with two singularities.
\end{theorem}

In Chapter~III, ``$P$-functions on manifolds'', round functions are constructed on two-dimensional manifolds, and examples of $P$-functions are given on three-dimensional manifolds and on manifolds of higher dimension.

We point out that all two-dimensional smooth compact connected manifolds are classified according to whether they admit round functions.

\begin{theorem}[3.2]
On a manifold homeomorphic to the sphere $S^2$, the projective plane $\R P^2$, or the disc $D^2$, there are no round functions.
On any other two-dimensional smooth compact connected manifold round functions exist.
\end{theorem}

\begin{theorem}[3.3]
On any two-dimensional smooth compact connected manifold not homeomorphic to the disc $D^2$, sphere $S^2$ or projective plane $\R P^2$ there exists an exact round function whose number of singularities equals:
\begin{enumerate}\setlength{\itemsep}{0pt}
\item two on the torus and the Klein bottle;
\item three on the remaining manifolds without boundary;
\item one on the annulus $S^1 \times I$ and on the Möbius band;
\item two on the remaining manifolds with boundary.
\end{enumerate}
\end{theorem}

On three-dimensional manifolds round functions are studied.
Among such manifolds we distinguish those on which an exact round function exists.

\begin{theorem}[3.4]
On every three-dimensional smooth compact connected manifold without boundary a round function exists.
If the manifold is homeomorphic to one of
\begin{enumerate}\setlength{\itemsep}{0pt}
\item[a)] the sphere $S^3$,
\item[b)] the projective space $\R P^3$,
\item[c)] the manifold $S^1 \times S^2$,
\item[d)] a lens space,
\end{enumerate}
then it admits an exact round function with two critical circles.
If the manifold is homeomorphic to the torus $T^3$, then it admits an exact round function with three critical circles.
On all other three-dimensional smooth compact connected manifolds there exists a round function whose number of singularities does not exceed four.
\end{theorem}

On $n$-dimensional manifolds we consider $P$-functions for which $P$ is a sphere.

\begin{theorem}[3.5]
Let $P = S^{n-1}$ be the $(n-1)$-dimensional sphere.
Then on the $n$-dimensional sphere there are no $P$-functions.
\end{theorem}

\begin{theorem}[3.6]
Let $P = S^k$ be the $k$-dimensional sphere, and let
\[
   n \geq 2k + 1.
\]
Then on the $n$-dimensional sphere there exists a $P$-function whose number of singularities does not exceed $[n/2] + 1$.
\end{theorem}

The author wishes to express her deep gratitude to her scientific supervisor Volodymyr Vasylovych Sharko, communication with whom made it possible to broaden and deepen her knowledge of contemporary problems of mathematics, and to whom in particular belongs the idea of considering this topic, and also to Yevhen Mykhailyuk for a number of interesting ideas and remarks.

\chapter{$P$-category of a topological space}

The notion of Lyusternik--Shnirelman category was introduced by its authors for estimating the number of critical values, and hence of critical points, of a function on a manifold, and was used in problems of the calculus of variations.
This notion was developed and applied by such mathematicians as K.\,Borsuk~\cite{borsuk-1936}, R.\,Fox~\cite{fox-1941}, T.\,Ganea~\cite{ganea-1967}, F.\,Takens~\cite{takens-1968}, W.\,Singhof~\cite{singhof-1979}, Chogoshvili, Frolov and Elsgolts \cite{frolov-elsgolts-1939}, and others.
Their work was directed toward further investigation of the relationship between the topology of a space and its category (the Lyusternik--Shnirelman category, the strong category, etc.).
For instance, it was proved~\cite{ganea-1967} that any $(n-1)$-connected CW-complex of dimension at most $3n-3$ and of Lyusternik--Shnirelman category at most one has the homotopy type of a suspension.
Frolov and Elsgolts \cite{frolov-elsgolts-1939} related the notion of Lyusternik--Shnirelman category of a manifold to its homological characteristic --- the length of the manifold.
In modern mathematics one uses the notion of cohomological length of a manifold, which is applicable to a wider class of objects.

This chapter is devoted to the generalisation of the well-known inequality between the length of a manifold and its Lyusternik--Shnirelman category
\[
   \cat M \;\geq\; \mathrm{long}\, M + 1
\]
to the case where the categorical subsets are contracted to homeomorphic submanifolds.
We also give examples of the computation of such a category and of the correspondingly generalised length of a manifold.

\section{Definition and properties of the $P$-category of a topological space}

\begin{definition}[1.1]\label{def:1.1}
Let $A$, $B$, $P$ be closed subsets of a Hausdorff topological space, with $P \subset B$ and $A \subset B$.
The \emph{$P$-category} $\Pcat_B A$ of the set $A$ relative to the set $B$ is the minimal number $k$ of closed subsets $A_1, \ldots, A_k$ in $B$ that satisfy the following properties:
\begin{enumerate}\setlength{\itemsep}{0pt}
\item $A$ is their union
      \[
         A = \bigcup_{i=1}^{k} A_i;
      \]
\item each subset $A_i$, $i=1,\ldots,k$, is contractible within $B$ to $P$.
That is, for each subset $A_i$ there exists a homotopy $F_i\colon A_i \times I \to B$ such that
      \begin{itemize}\setlength{\itemsep}{0pt}
\item[a)] $F_i(x,0) = x$, $x \in A_i$;
\item[b)] $F_i(x,t) \in B$, $x \in A_i$, $0 < t < 1$;
\item[c)] $F_i(x,1) = P_i \subset B$, where $P_i$ is homeomorphic to $P$.
      \end{itemize}
\end{enumerate}
If no such number $k$ exists, we set $\Pcat_B A = \infty$.
The subsets $A_1, \ldots, A_k$ are called \emph{categorical}.
The $P$-category of a set $A$ relative to itself is simply called the $P$-category of $A$ and denoted $\Pcat A$.
If $A$ coincides with the whole topological space, then $\Pcat A$ denotes the $P$-category of that space.
\end{definition}

Observe that when $P$ is a point, the definition of $P$-category coincides with the definition of the Lyusternik--Shnirelman category.

\begin{definition}[1.2]\label{def:1.2}
The $P$-category of a topological space is called the \emph{round category} if $P$ is the circle $S^1$.
\end{definition}

\subsection*{Properties of the $P$-category}

\paragraph{Property 1.}
Let $A$, $B$, $P$ be closed subsets of a closed subset $C$ of a Hausdorff topological space.
Then
\[
   \Pcat_C(A \cup B) \;\leq\; \Pcat_C A + \Pcat_C B.
\]

\textit{Proof.} Let $\Pcat_C A = k$ and $\Pcat_C B = p$, that is, there exist closed subsets $A_1, \ldots, A_k$ and $B_1, \ldots, B_p$ such that
\[
   A = \bigcup_{i=1}^{k} A_i, \quad B = \bigcup_{i=1}^{p} B_i,
\]
and each of them is contracted within $C$ to $P$.
Their union is a covering of $A \cup B$.
The required inequality then follows from the minimality of the number of such subsets in the definition of $P$-category.

\paragraph{Property 2.}
Let $A$, $B$, $P$ be closed subsets of a closed subset $C$ of a Hausdorff topological space, with $B \subset C$.
Then
\[
   \Pcat_B A \;\geq\; \Pcat_C A.
\]

\textit{Proof.}
Let $\Pcat_B A = k$, that is, $A$ is covered by sets $A_1, \ldots, A_k$, each of which is contracted to $P$ within $B$.
Since $B \subset C$, each $A_i$ may also be considered as contracted within $C$; the required inequality follows from the minimality of the number of such sets in the definition of $P$-category.

\paragraph{Property 3.}
Let $A$, $B$, $P$ be closed subsets of a closed subset $C$ of a Hausdorff topological space, and let $A$ be a deformation retract of $B$.
Then
\[
   \Pcat_C B \;\leq\; \Pcat_C A.
\]

\textit{Proof.}
Let $\Pcat_C A = k$, that is, there is a covering of $A$ by closed subsets $A_i$, $i=1,\ldots,k$, each of which is contracted within $C$ to $P$.
The fact that $A$ is a deformation retract of $B$ means that there is a homotopy $F\colon B \times I \to B$ such that
\begin{itemize}\setlength{\itemsep}{0pt}
\item[a)] $F(x,0) = x$, $x \in B$;
\item[b)] $F(x,t) = x$, $x \in A$, $0 \leq t \leq 1$;
\item[c)] $F(x,1) \in A$, $x \in B$.
\end{itemize}
Cover $B$ by the subsets
\[
   B_i = \{ x \in B : F(x,1) \in A_i \}.
\]
Each $B_i$ is contracted within $C$ to $P$: first via the homotopy $F$ to $A_i$, then to $P$ by the homotopy that contracts $A_i$ to $P$.
The required inequality follows from the minimality of the number of such subsets in the definition of $P$-category.

\paragraph{Property 4.}
Let $A$ and $P$ be closed subsets of a manifold $M$.
The $P$-category of $A$ relative to $M$ admits the following lower bound:
\[
   \Pcat_M A \;\geq\; \frac{\mathrm{long}_M A + 1}{\mathrm{long}_M P + 1},
\]
where $\mathrm{long}_M A$ and $\mathrm{long}_M P$ are the lengths \cite{frolov-elsgolts-1939} of $A$ and $P$ in $M$.

\textit{Proof.}
Let $\Pcat_M A = k$, that is, there is a covering of $A$ by closed subsets $A_1, \ldots, A_k$, each contracted within $M$ to $P$.
By Property~3 of length \cite[p.\,565]{frolov-elsgolts-1939}, the length of each categorical subset does not exceed the length of $P$ in $M$:
\[
   \mathrm{long}_M A_i \;\leq\; \mathrm{long}_M P, \quad i=1,\ldots,k.
\]
From this inequality and the properties of length~\cite{frolov-elsgolts-1939} we obtain
\[
   \mathrm{long}_M A
   = \mathrm{long}_M \bigcup_{i=1}^{k} A_i
   \;\leq\; \sum_{i=1}^{k} \mathrm{long}_M A_i + k - 1
   \;\leq\; k\cdot\mathrm{long}_M P + k - 1,
\]
from which the required inequality follows.

\begin{corollary}
If, with the notation above, $A$ coincides with $M$ and $P$ is a closed submanifold of $M$, then
\[
   \Pcat M \;\geq\; \frac{\mathrm{long}\, M + 1}{\mathrm{long}\, P + 1}.
\]
\end{corollary}

In particular, when $P$ is a point, the well-known estimate of the Lyusternik--Shnirelman category of $M$
\[
   \cat M \;\geq\; \mathrm{long}\, M + 1
\]
coincides with the bound above.

\section{$P$-length of a manifold}

\begin{definition}[1.3]\label{def:1.3}
Let $m$ be the dimension of a manifold $M$ and let $p$ be a non-negative integer.
Let $H^*(M; A)$ be the cohomology ring of $M$, where $A = \Z$ or $\R$ when $M$ is orientable, and $A = \Z_2$ when $M$ is non-orientable.
Consider all integers $q$ such that in the cohomology ring $H^*(M; A)$ there exist elements $a_1, \ldots, a_q$ satisfying:
\begin{enumerate}\setlength{\itemsep}{0pt}
\item[1)] $\dim a_i > 0$, $i = 1, \ldots, q$;
\item[2)] the product $a_1 \cdots a_q$ (with respect to the multiplication in $H^*(M; A)$) is non-zero;
\item[3)] $\dim (a_1 \cdots a_q) \leq m - p$.
\end{enumerate}
The \emph{cohomological $p$-length} of $M$ is the largest such number $q$.
\end{definition}

Since for connected compact manifolds Poincaré duality gives an isomorphism
\[
   D\colon H^k(M) \to H_{n-k}(M)
\]
between cohomologies and homologies, we formulate the dual definition of $p$-length.

\begin{definition}[1.3$'$]\label{def:1.3prime}
With the notation of Definition~\ref{def:1.3}, the \emph{homological $p$-length} of $M$ is the largest number $q$ of cycles $c_1, \ldots, c_q$ of $M$ satisfying:
\begin{enumerate}\setlength{\itemsep}{0pt}
\item[$1')$] $\dim c_i < m$, $i=1,\ldots,q$;
\item[$2')$] $\bigcap_{i=1}^{q} c_i = c$ is a cycle not homologous to zero;
\item[$3')$] $\dim c \geq p$.
\end{enumerate}
\end{definition}

The (co)homological $p$-length of $M$ is called the \emph{$p$-length} of $M$ and is denoted by $\mathrm{long}^p M$.

\begin{remark}[1.4]\label{rem:1.4}
None of the cycles $c_1, \ldots, c_q$ in the definition of $p$-length may contain another, otherwise the definition would lose meaning, since one would have to consider only manifolds of infinite length.

If $M$ has a boundary, then in the definition of $p$-length only absolute cycles of $M$ are considered.
\end{remark}

\begin{theorem}[1.5]\label{thm:1.5}
If $M$ is a manifold without boundary, then its $P$-category is at least one more than the $p$-length of $M$:
\[
   \Pcat M \;\geq\; \mathrm{long}^{\,p} M + 1.
\]
If $M$ is a manifold with boundary, then its $P$-category is at least the $p$-length of $M$:
\[
   \Pcat M \;\geq\; \mathrm{long}^{\,p} M.
\]
\end{theorem}

\textit{Proof.} Consider cycles $c_1, \ldots, c_q$ in $H_*(M)$ of dimension less than the dimension $m$ of $M$, such that their intersection is a cycle not homologous to zero of dimension at least $p$, and
\[
   q = \mathrm{long}^{\,p} M.
\]

Suppose
\[
   \Pcat M = s \;\leq\; q.
\]
Then there exist closed subsets $A_1, \ldots, A_s$ in $M$, contracting to $P$, whose union is $M$.
Without loss of generality we may assume that $s = q$ and
\[
   M = \bigcup_{i=1}^{q} A_i.
\]
It suffices, if $s < q$, to add $q - s$ submanifolds $P$ to the collection $\{A_i\}_{i=1}^{s}$.
If $q = 1$, that is, $\mathrm{long}^{\,p} M = 1$, then conditions $1')$ and $3')$ of Definition~\ref{def:1.3prime} imply that $p < m$.
The $P$-category of $M$ cannot equal one, since this would mean that $M$ contracts within itself to a subset of dimension $p$ less than $m$.

Hence $\Pcat M$ must be at least two, that is, in the case $q = 1$ the statement of the theorem holds.
From now on we assume $q \geq 2$.
In this case condition $3')$ of the definition of $p$-length implies that each cycle $c_1, \ldots, c_q$ has dimension greater than $p$:
\[
   \dim c_i > p.
\]

Associate to each cycle $c_i$ the set $A_i$.
Since each $A_i$ contracts to $P$, the exact homology sequence of the pair
\[
   \ldots \to H_l(A_i) \xrightarrow{i_*} H_l(M)
   \xrightarrow{j_*} H_l(M, A_i) \to \ldots
\]
gives, for $l > p$, $i_* = 0$, that is, $j_*$ is a monomorphism, which is equivalent to the equality
\[
   H_l(M) = H_l(M, A_i) \quad \text{for } l > p.
\]
This means that each cycle $c_i$ is homologous to a cycle
\[
   \gamma_i \subset H_l(M, A_i),
\]
that is, the support of $\gamma_i$ lies in $M \setminus A_i$.
In other words, each cycle $\gamma_i$ can be ``moved off'' the subset $A_i$.
Moving off all cycles $\gamma_i$ in this way and taking their intersection
\[
   0 \neq c = c_1 \cap \ldots \cap c_q \sim \gamma_1 \cap \ldots \cap \gamma_q = \gamma,
\]
we obtain that the non-zero cycle $\gamma$ would have to belong to the empty set, since
\[
   \gamma \subset (M \setminus A_1) \cap \ldots \cap (M \setminus A_q)
   = M \setminus \bigcup_{i=1}^{q} A_i = \emptyset.
\]

In the case of a manifold with boundary, suppose
\[
   \Pcat M = s \;\leq\; q - 1.
\]
Without loss of generality $s = q - 1$ and
\[
   M = \bigcup_{i=1}^{q-1} A_i.
\]
Analogously to the previous case, we move $q-1$ cycles $c_i$ off the subsets $A_i$, $i = 1, \ldots, q-1$.
The cycle
\[
   \gamma = \bigcap_{i=1}^{q-1} \gamma_i
\]
must then lie in the boundary $\partial M$.
But the absolute cycle $c_q$ and the cycle $\gamma$, placed in general position, do not intersect:
\[
   0 \neq c = c_1 \cap \ldots \cap c_{q-1} \cap c_q
   \sim \gamma_1 \cap \ldots \cap \gamma_{q-1} \cap c_q
   = \gamma \cap c_q = \emptyset,
\]
that is, a non-zero cycle would have to lie in the empty set.
This contradiction completes the proof.
\qed

Let us give examples of computing the $p$-length of manifolds.

\begin{proposition}[1.6]\label{prop:1.6}
Let $(N; \partial N)$ be obtained from the sphere $S^2$ by removing a finite number $s$ of open discs:
\[
   (N; \partial N) = \bigl(S^2 \setminus \bigcup_{i=1}^{s} D_i^2;\; \bigcup_{i=1}^{s} S_i^1\bigr),
\]
and let $(L; \partial L)$ be obtained from the projective plane $\R P^2$ by removing a finite number $k$ of open discs:
\[
   (L; \partial L) = \bigl(\R P^2 \setminus \bigcup_{i=1}^{k} D_i^2;\; \bigcup_{i=1}^{k} S_i^1\bigr).
\]
For $p = 0$ the $p$-length of a two-dimensional compact connected manifold equals:
\begin{enumerate}\setlength{\itemsep}{0pt}
\item one, if the manifold is homeomorphic to $(N; \partial N)$, to $(L; \partial L)$, or to the sphere $S^2$;
\item two, on the other manifolds.
\end{enumerate}
For $p = 1$ the $p$-length of a two-dimensional compact connected manifold equals:
\begin{enumerate}\setlength{\itemsep}{0pt}
\item zero, if the manifold is homeomorphic to the sphere $S^2$ or to the disc $D^2$;
\item one, on the other manifolds.
\end{enumerate}
\end{proposition}

\textit{Proof.}
Let $M$ be the manifold given in the statement.
All cycles considered on $M$ are taken to be realised as simplicial subcomplexes in general position.

If $p = 0$, then the cycles $c_1, \ldots, c_q$, where $q$ equals the $p$-length, must be chosen among the zero-dimensional or one-dimensional cycles (by Property $(1)'$ of $p$-length).
If a zero-dimensional cycle appears among them, $q$ must equal one by Remark~\ref{rem:1.4}.
Hence
\[
   \mathrm{long}^{\,p} M \;\geq\; 1.
\]
In order to possibly increase the number of cycles satisfying the conditions of $p$-length, we consider one-dimensional cycles.

\begin{figure}[h]
\centering\begin{tabular}{ccc}
\includegraphics[width=4cm]{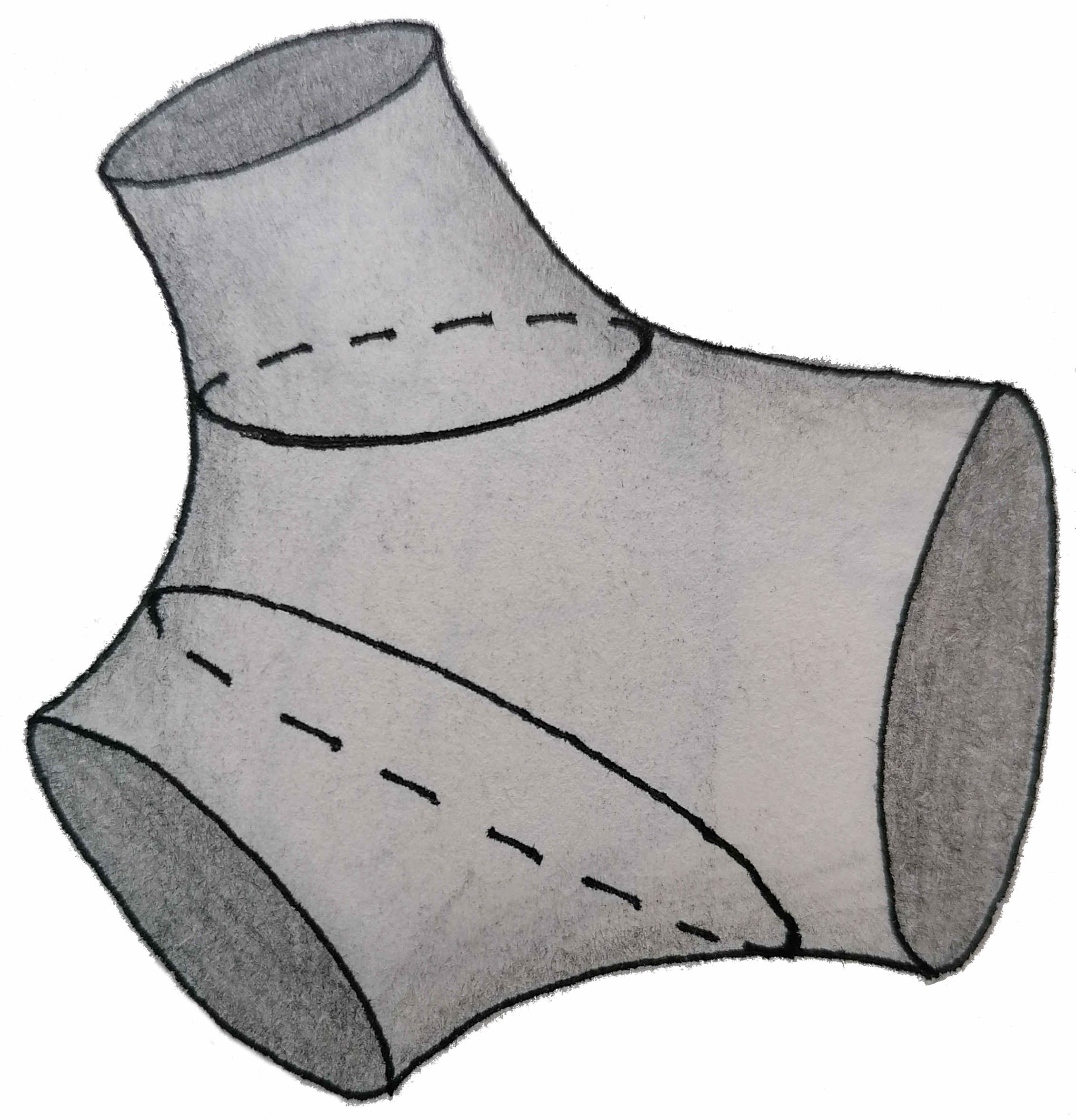} & \qquad\qquad &
\includegraphics[width=4cm]{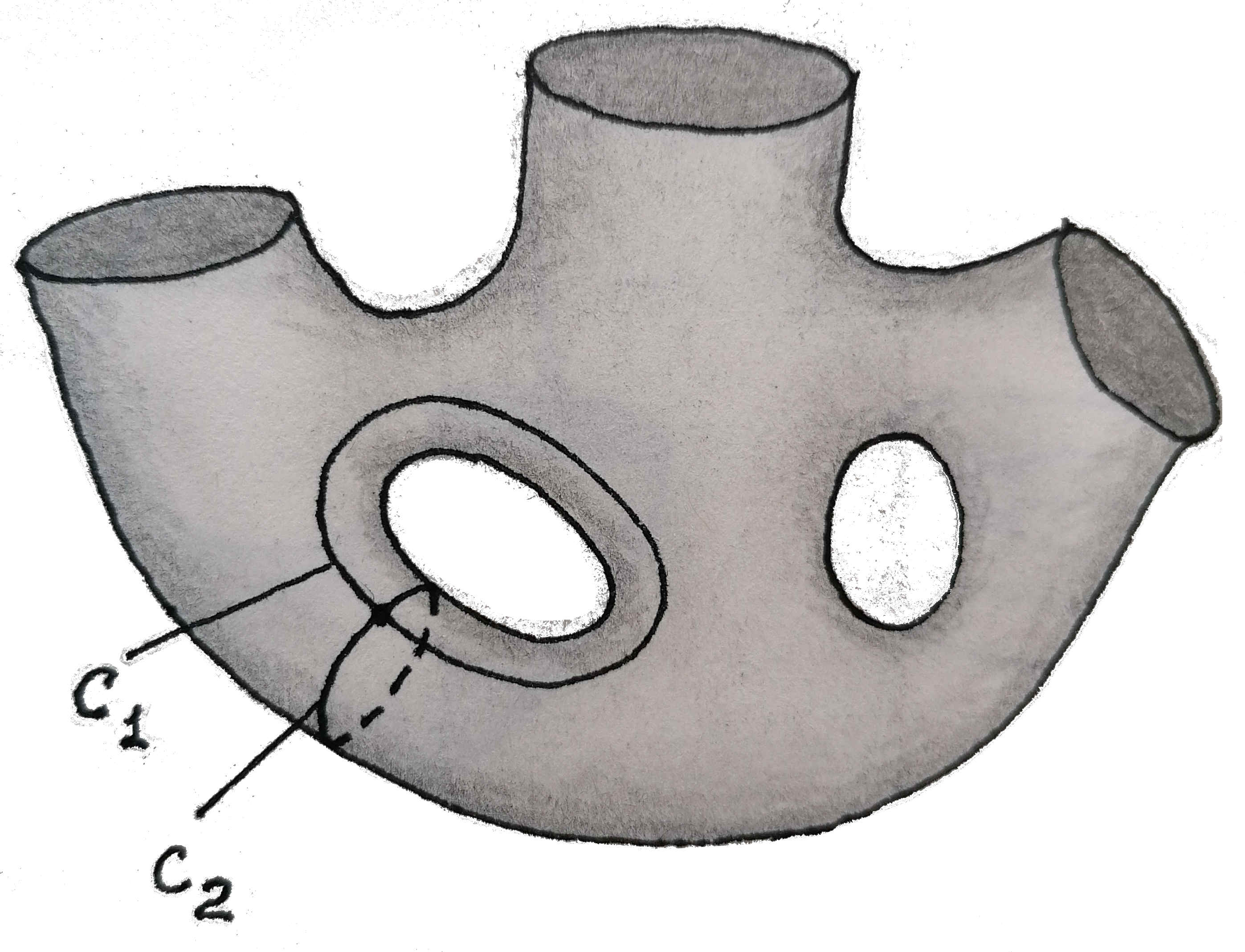} \\
(a) & & (b)
\end{tabular}
\caption{}\label{fig:1.1}
\end{figure}

In the first case, where $M$ is homeomorphic to $(N; \partial N)$ or $(L; \partial L)$, one-dimensional cycles in general position do not intersect (Fig.~1.1(a)), so $q = 1$.
On the other manifolds with boundary, the maximal number of one-dimensional cycles whose intersection is a cycle not homologous to zero equals two (Fig.~1.1(b)), that is, $q = 2$.

If $M$ is homeomorphic to the sphere $S^2$, then $H_1(M) = 0$, so the cycle $c$ in Definition~\ref{def:1.3prime} must be taken to be a zero-dimensional cycle.
This means $q = 1$.

If $M$ is a manifold without boundary which is not homeomorphic to $S^2$, then by the classification theorem of two-dimensional manifolds (see, e.g., Massey) it is homeomorphic either to a sphere $S^2$ with $k$ handles ($k > 0$) or to a sphere $S^2$ with $s$ Möbius bands ($s > 0$).
In this case the maximal number of one-dimensional cycles whose intersection is not homologous to zero equals two (recall that the cycles are in general position).

Figure~1.2 shows possible cycles in the orientable case (a) and non-orientable case (b).

\begin{figure}[h]
\centering\begin{tabular}{ccc}
\includegraphics[width=4cm]{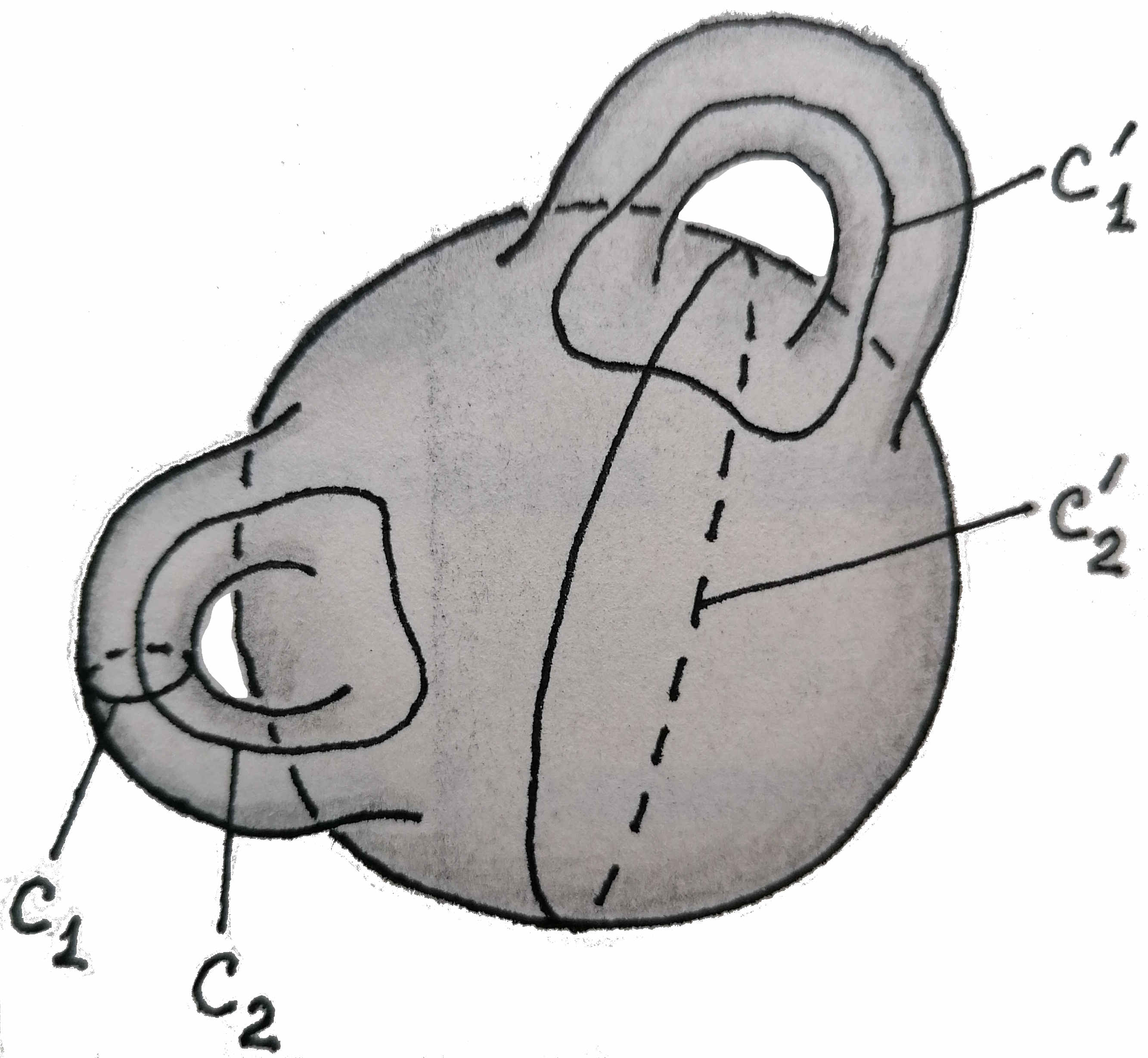} & \qquad\qquad &
\includegraphics[width=6cm]{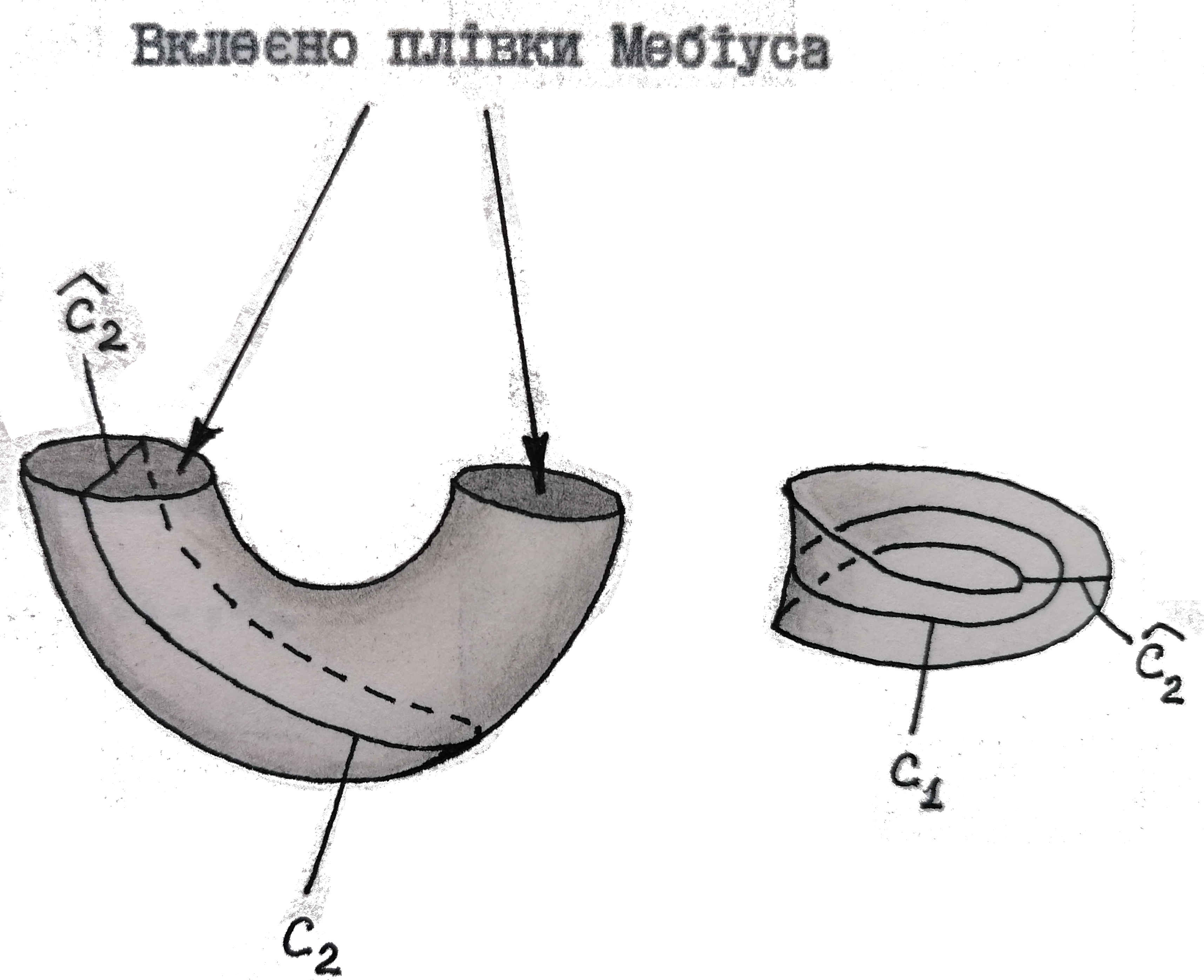} \\
(a) & & (b)
\end{tabular}
\caption{}\label{fig:1.2}
\end{figure}

If $p = 1$, then the cycles $c_1, \ldots, c_q$, where $q = \mathrm{long}^{\,p} M$, must be taken to be one-dimensional by the properties of $p$-length.
Moreover, their number must not exceed one; otherwise the intersection of the cycles could not have dimension greater than or equal to one.
Thus the $p$-length of the manifold equals one if there exists a one-dimensional cycle not homologous to zero on the manifold, and zero otherwise.
Since
\[
   H^1(S^2) = 0 \quad \text{and} \quad H_1(M) \neq 0
\]
for all other two-dimensional manifolds $M$ without boundary, the required statement follows.
\qed

\begin{theorem}[1.7]\label{thm:1.7}
Let
\[
   M = S^{k_1} \times S^{k_2} \times \ldots \times S^{k_r}
\]
be a product of $k_i$-dimensional spheres, $i = 1, \ldots, r$, and let $p$ be a non-negative integer.
Then
\[
   \mathrm{long}^{\,p} M \;\geq\; r - \min l(p),
\]
where the minimum is taken over all subsets
\[
   \bigl\{(k_{i_1}, k_{i_2}, \ldots, k_{i_l}) \subset (k_1, k_2, \ldots, k_r) :
      \textstyle\sum_{j=1}^{l} k_{i_j} \geq p\bigr\}.
\]
\end{theorem}

\textit{Proof.}
All cycles considered are in general position and realised as simplicial complexes.
From the set
\[
   \{k_1, \ldots, k_r\}
\]
choose $l$ elements
\[
   k_{i_1}, k_{i_2}, \ldots, k_{i_l},
\]
whose sum is at least $p$.
The corresponding submanifold
\[
   Q = S^{k_{i_1}} \times S^{k_{i_2}} \times \ldots \times S^{k_{i_l}}
\]
may be viewed as a cycle of dimension
\[
   \sum_{j=1}^{l} k_{i_j}.
\]
Multiplying $Q$ by each of the remaining factors of the product
\[
   S^{k_1} \times S^{k_2} \times \ldots \times S^{k_r},
\]
we obtain $r - l$ cycles $c_1, \ldots, c_{r-l}$ of $M$ satisfying:
\begin{enumerate}\setlength{\itemsep}{0pt}
\item[1)] $\dim c_i < n = \sum_{i=1}^{r} k_i$, if $l + 1 \leq r$;
\item[2)] $\dim \bigcap_{i=1}^{r-l} c_i \geq p$, since $\bigcap_{i=1}^{r-l} c_i = Q$;
\item[3)] $\bigcap_{i=1}^{r-l} c_i = Q$ is a cycle not homologous to zero.
\end{enumerate}
Thus the cycles $c_1, \ldots, c_{r-l}$ satisfy conditions $1)'$--$3)'$ of the definition of $p$-length.
Taking the minimal number $l$ of elements from $\{k_1, \ldots, k_r\}$ having the required properties, the inequality
\[
   \mathrm{long}^{\,p} M \;\geq\; r - \min l(p)
\]
holds.
Then Theorem~\ref{thm:1.5} together with the estimate of $\Pcat M$ from Theorem~\ref{thm:1.14}
\[
   \Pcat M \;\leq\; r - l + 1
\]
yields the statement of the theorem.
\qed

\begin{corollary}[1.8]\label{cor:1.8}
Let
\[
   M = S^{k_1} \times S^{k_2} \times \ldots \times S^{k_r}
\]
be a product of $k_i$-dimensional spheres, $i = 1, \ldots, r$.
Then the (co)homological length of $M$ equals the number $r$ of factors:
\[
   \mathrm{long}\, M = r.
\]
\end{corollary}

\textit{Proof.}
Since $p = 0$, we may take as cycles forming the length of $M$ the spheres $S^{k_j}$, $j = 1, \ldots, r$.
Their intersection is a single point --- a cycle of dimension $0$ which is not homologous to zero.

The following proposition is an immediate consequence of Theorem~\ref{thm:1.7}.

\begin{proposition}[1.9]\label{prop:1.9}
The $p$-length of the $n$-dimensional torus is $p$ less than its dimension:
\[
   \mathrm{long}^{\,p} T^n = n - p.
\]
\end{proposition}

\section{Examples of computation of the $P$-category of manifolds}

\begin{proposition}[1.10]\label{prop:1.10}
The Lyusternik--Shnirelman category of a two-dimensional compact connected manifold with boundary, not homeomorphic to the disc $D^2$, equals two.
Moreover, the categorical sets may be chosen so as to intersect only along their boundaries, and so that each is homeomorphic to the disc $D^2$.
\end{proposition}

\textit{Proof.}
Since any two-dimensional compact connected manifold $M$ with boundary is homeomorphic to the sphere $S^2$ with a certain number of ``holes'' forming the boundary and a certain number of handles or Möbius bands, the covering of $M$ by two discs may be exhibited as follows.
An orientable manifold with a single hole is covered by two discs as shown in Fig.~1.3(a).

If the boundary of the manifold has several components, then the first categorical set is combined with the set that ``cuts the holes'' (Fig.~1.3(b)), so that the union is homeomorphic to a disc and the remaining set is also homeomorphic to the disc $D^2$.

\begin{figure}[h]
\centering\includegraphics[width=0.85\textwidth]{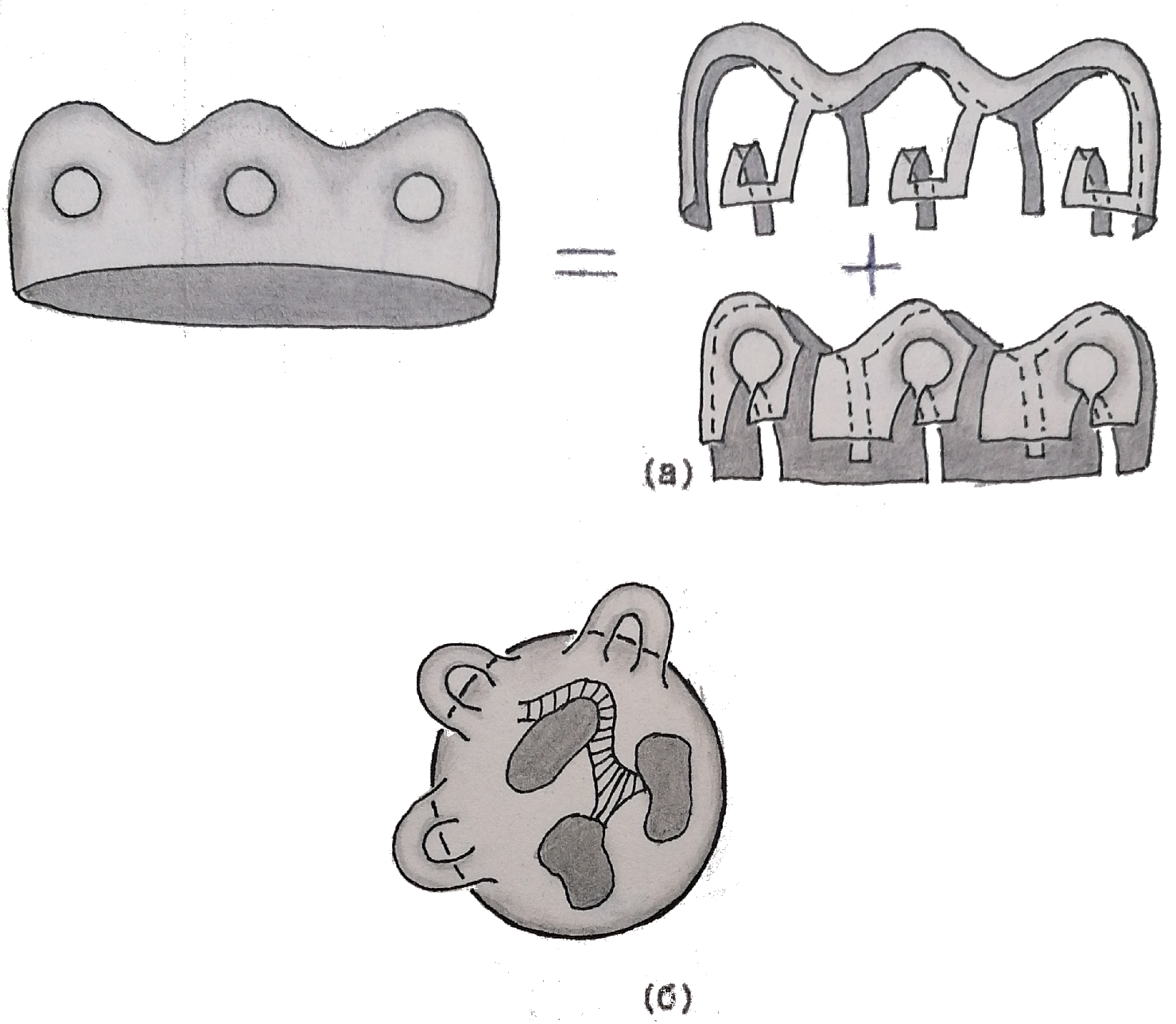}
\caption{}\label{fig:1.3}
\end{figure}

In the non-orientable case, a manifold with one boundary component is covered by the two required subsets as shown in Fig.~1.4.
Here we depict the sphere $S^2$ with three Möbius bands glued in and one ``hole''.

As the first categorical set we take the parts of the Möbius bands (Fig.~1.4(b)) glued into the ``holes'' 1, 2, 3, combined with the shaded portion of the manifold (Fig.~1.4(a)).
If the non-orientable manifold has more than one boundary component, then the first categorical set is combined with the set that ``cuts the holes'' (Fig.~1.3(b)) as is done for orientable manifolds.
This union is a subset homeomorphic to the disc $D^2$.
If this subset is removed from the manifold, the closure of the remaining subset is also homeomorphic to the disc $D^2$.

\begin{figure}[h]
\centering\begin{tabular}{ccc}
\includegraphics[width=6cm]{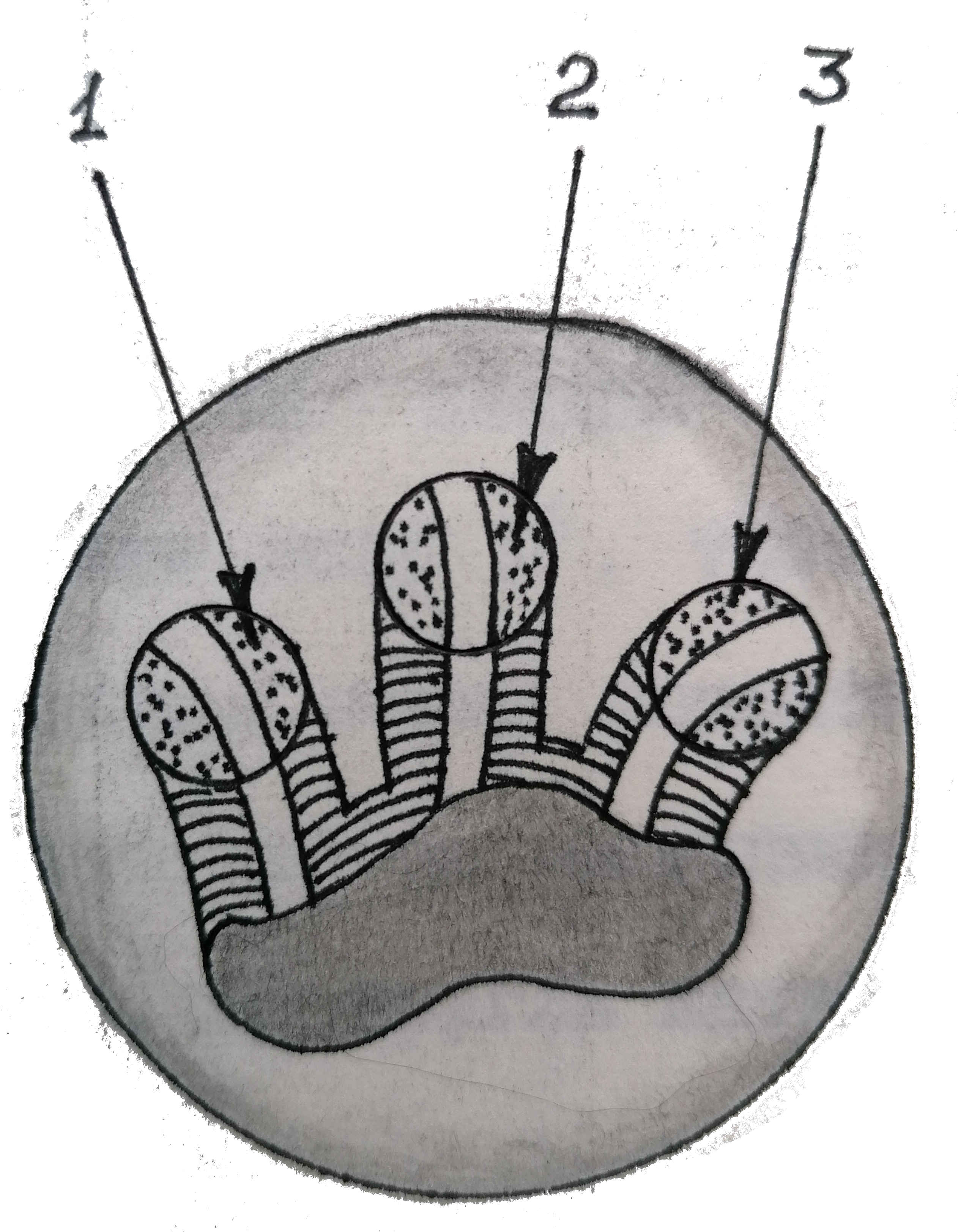} & \qquad\qquad &
\includegraphics[width=5cm]{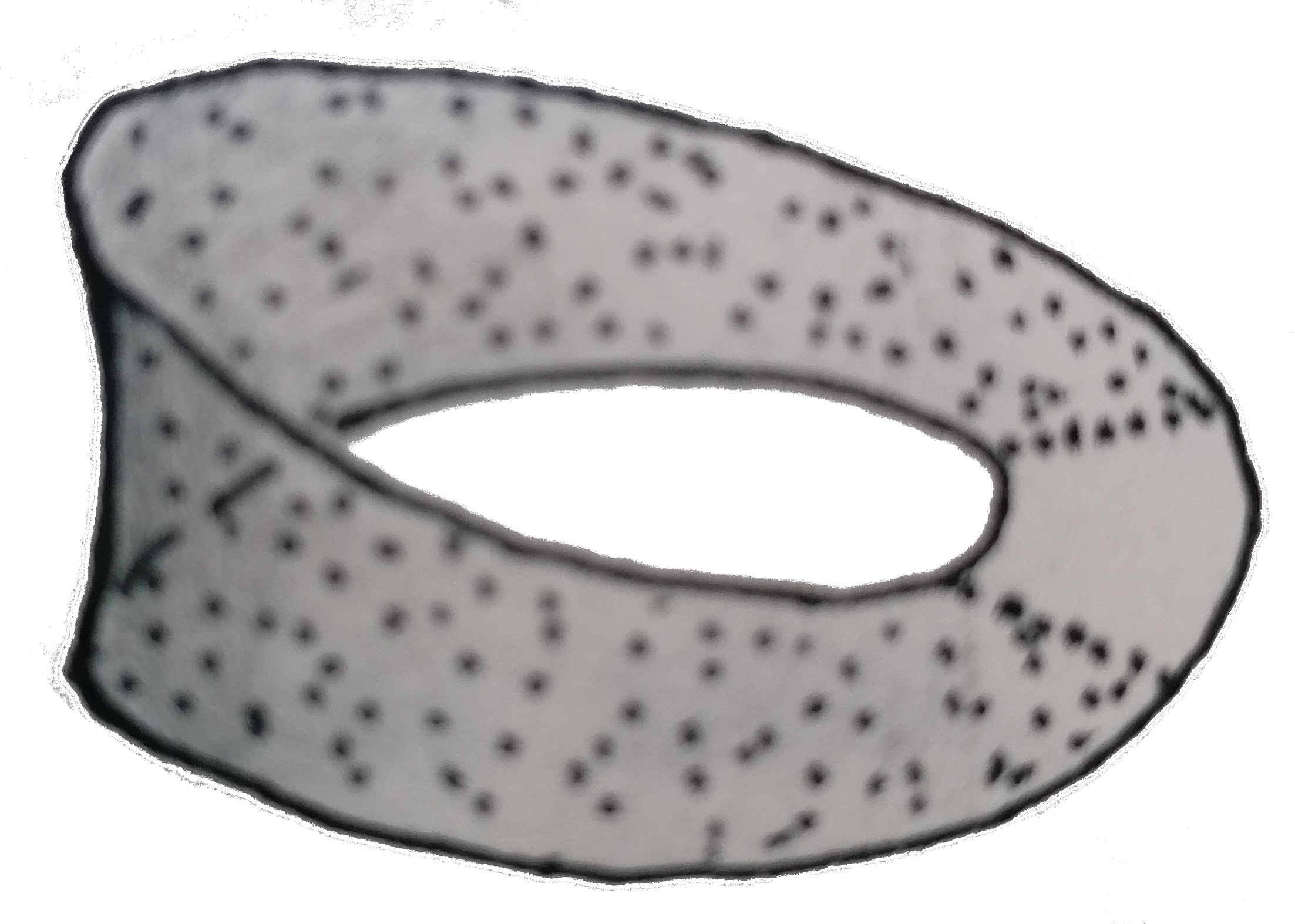} \\
(a) & & (b)
\end{tabular}
\caption{}\label{fig:1.4}
\end{figure}

\begin{proposition}[1.11]\label{prop:1.11}
The round category of a two-dimensional compact connected manifold equals:
\begin{enumerate}\setlength{\itemsep}{0pt}
\item one, if the manifold is homeomorphic to $S^1 \times I$ or to the Möbius band;
\item two, on the remaining manifolds with boundary;
\item two, if the manifold is homeomorphic to the sphere $S^2$, the projective plane $\R P^2$, the torus $T^2$ or the Klein bottle;
\item three, on the remaining manifolds without boundary.
\end{enumerate}
\end{proposition}

\textit{Proof.} In the first case the statement is clear.

In the second case, if the manifold with boundary is not homeomorphic to $S^1 \times I$ or to the Möbius band, then its round category cannot equal one (this is prevented by the homologies of the manifold).
Hence the round category must be at least two.
We show that it equals exactly two.
Represent the orientable manifold as a sphere with a number of ``holes'' (corresponding to the boundary components) and a number of handles.
Suppose the boundary is homeomorphic to a single circle $S^1$, so there is only one ``hole'' on the sphere.
Then the manifold can be covered by two closed subsets, each contracting (even within itself) to a circle (Fig.~1.5).
Therefore the round category of the manifold equals two.

\begin{figure}[h]
\centering\includegraphics[width=0.85\textwidth]{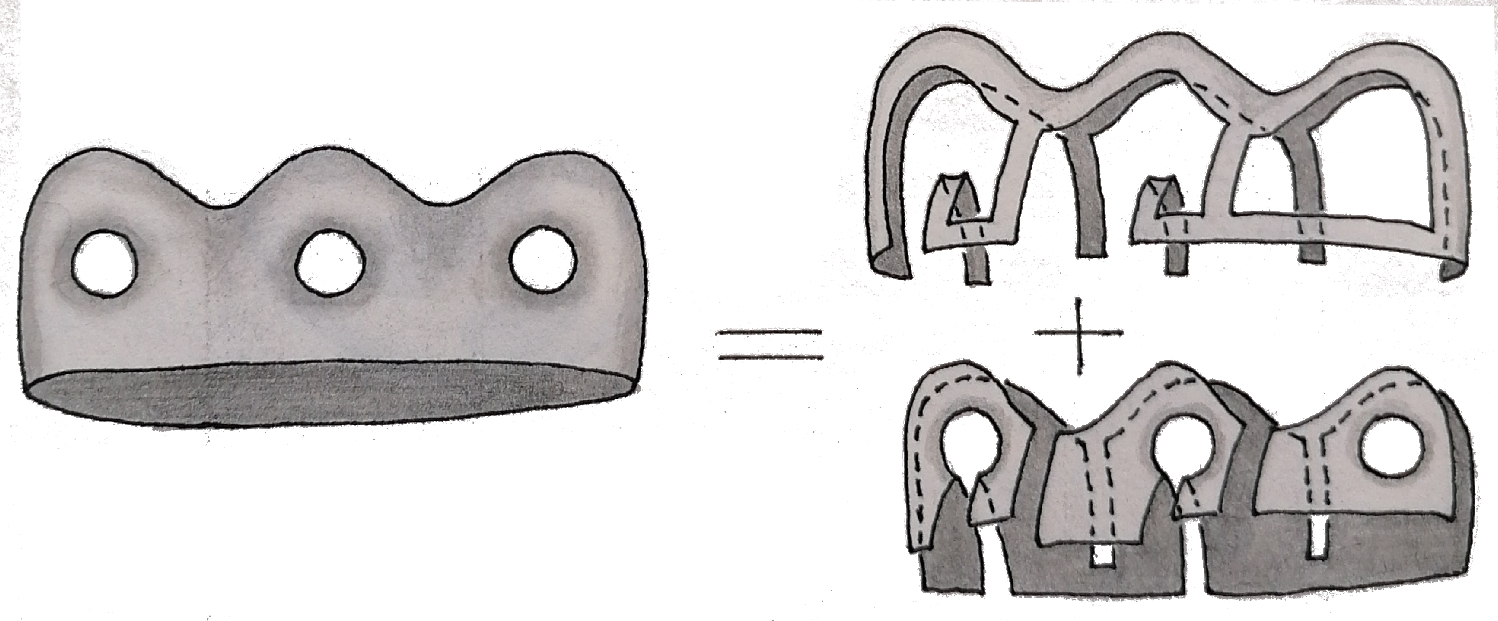}
\caption{}\label{fig:1.5}
\end{figure}

If the boundary consists of several components, then to the first categorical subset (Fig.~1.5) for the manifold with one boundary component we adjoin the subset that ``cuts the holes'' (Fig.~1.3(b)).
Thus the round category of such a manifold also equals two.

A non-orientable manifold is homeomorphic to a sphere with a number of Möbius bands ``patching the holes''.
The covering by two categorical sets is demonstrated by an example that easily generalises to any non-orientable manifold.
Figure~1.6 shows the covering of a sphere with three Möbius bands and three boundary components by two closed subsets, each contractible within itself to a circle.
The first subset consists of the entire Möbius band glued into ``hole'' 1, of parts of the Möbius bands (Fig.~1.6(b)) glued into ``holes'' 2 and 3, and of the subset that ``cuts the holes''.

\begin{figure}[h]
\centering\includegraphics[width=0.85\textwidth]{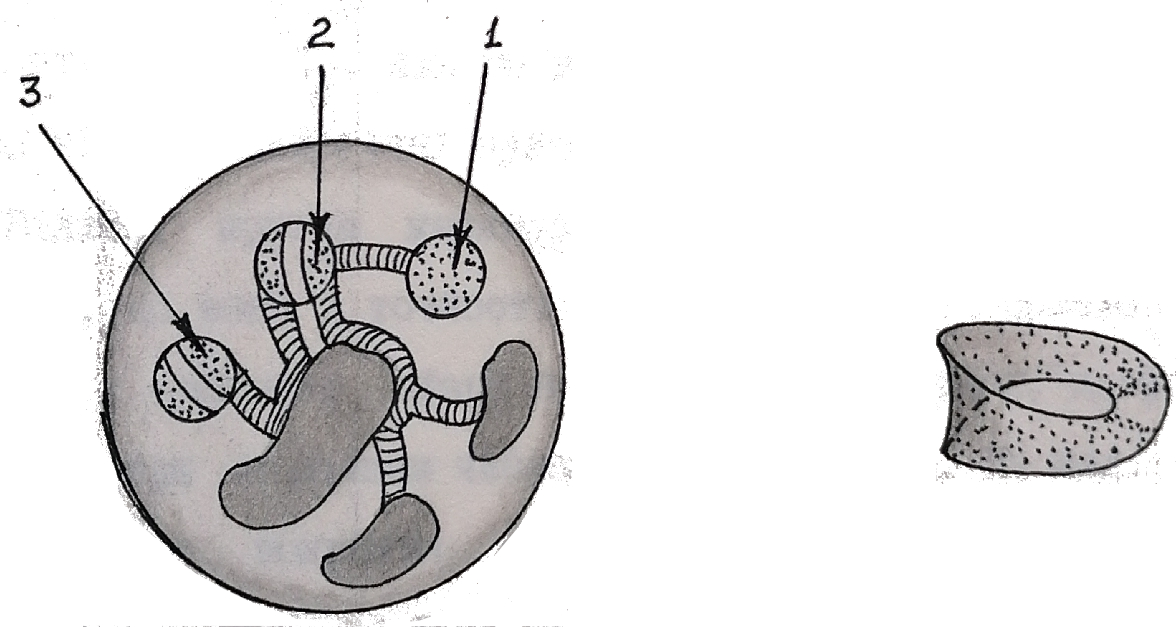}
\caption{}\label{fig:1.6}
\end{figure}

In the third case, where the manifold is homeomorphic to the sphere $S^2$, the projective plane $\R P^2$, the torus $T^2$ or the Klein bottle, the statement holds, since each such manifold can be covered by two closed subsets, each homeomorphic to the annulus $S^1 \times I$.

The round category of any other two-dimensional manifold without boundary cannot equal one (this is prevented by the manifold's homologies).
It also cannot equal two.
Suppose otherwise: there is a covering of the manifold by two closed sets, each contracting within the manifold to circles $S_1$ and $S_2$ not homologous to zero.
Without loss of generality $S_1$ and $S_2$ are not homologous to each other.
Since the rank of the one-dimensional homology of the manifold is at least four, there exist on the manifold circles $S_3$ and $S_4$ pairwise non-homologous to the circles $S_1$, $S_2$ and to each other.
Then there is a circle, non-homologous to $S_i$, $i = 1, \ldots, 4$, contained in at least one of the categorical sets $A_1$ (or $A_2$), which ``closes'' to a point; that is, through that point passes a non-zero cycle $S$ lying in $A_1$ and non-homologous to $S_1$.
But this is impossible, since $A_1$ may contain only circles non-homologous to zero that are homologous to $S_1$.

Therefore the round category of such a manifold is at least three.
We show that it equals exactly three.
Represent the orientable manifold as a sphere with a certain number of handles, and the non-orientable as a sphere with a certain number of Möbius bands.
Figure~1.7 displays a categorical covering of an orientable manifold by three closed subsets, each contractible within itself to a circle.

The covering of a non-orientable manifold by three closed sets, each contractible within itself to a circle, may be derived from Fig.~1.4 by combining the categorical sets shown there with a disc $D^2$ ``patching the hole'' and with the corresponding handle attached to the disc.

\begin{figure}[h]
\centering\includegraphics[width=0.85\textwidth]{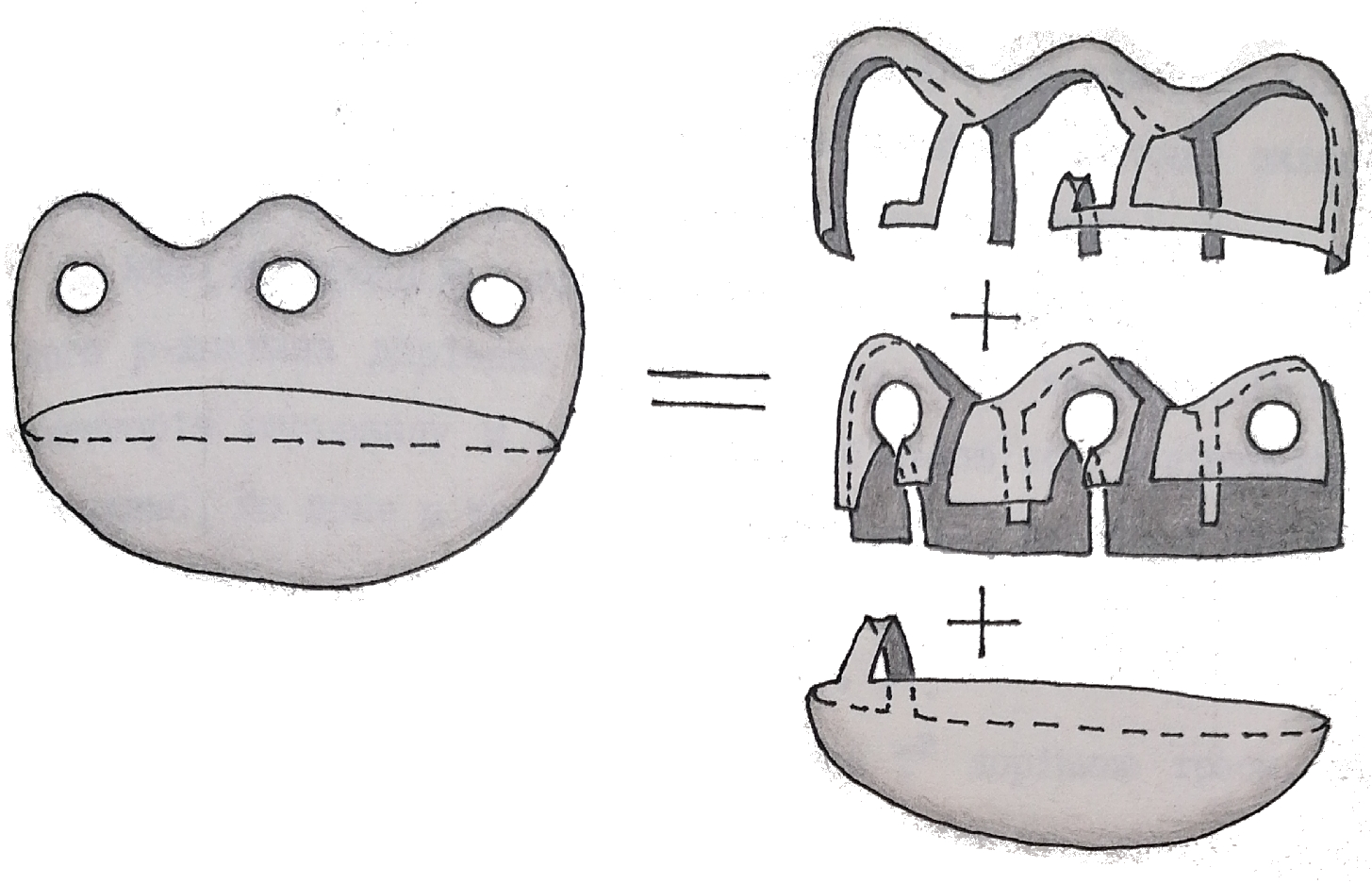}
\caption{}\label{fig:1.7}
\end{figure}

We note that the described coverings by categorical sets of round category are not unique.
The statement is proved.
\qed

\begin{proposition}[1.12]\label{prop:1.12}
The round category of a three-dimensional compact connected manifold without boundary
\begin{enumerate}\setlength{\itemsep}{0pt}
\item equals two if the manifold is homeomorphic to one of the following:
      \begin{itemize}\setlength{\itemsep}{0pt}
\item[a)] the sphere $S^3$;
\item[b)] $S^1 \times S^2$;
\item[c)] the projective space $\R P^3$;
\item[d)] a lens space obtainable as a quotient of $S^3$ by a differentiable action of the group $\Z_p$;
      \end{itemize}
\item equals three on the torus $T^3$;
\item does not exceed four on the remaining manifolds.
\end{enumerate}
\end{proposition}

\textit{Proof.}
In the first case, where the manifold is homeomorphic to one of a)--d), the statement is clear, since each of the listed manifolds has Heegaard genus one.
This means that the manifold may be presented as the gluing of two solid tori $S^1 \times D^2$ along some homeomorphism of their boundaries.

If $M$ is homeomorphic to the torus $T^3$, then by Proposition~\ref{prop:1.9} its $p$-length equals two (with $p=1$).
By Theorem~\ref{thm:1.5} the round category of $M$ must be at least three.
We show it equals exactly three.
Present the torus $T^3$ as
\[
   T^3 = S^1 \times T^2.
\]
The Lyusternik--Shnirelman category of the torus $T^2$ equals three, that is, $T^2$ can be covered by three closed subsets
\[
   T^2 = \bigcup_{i=1}^{3} A_i,
\]
each contracting on the torus to a point.
Then $T^3$ can be covered by three closed subsets
\[
   T^3 = S^1 \times T^2 = S^1 \times \bigl(\bigcup_{i=1}^{3} A_i\bigr)
   = \bigcup_{i=1}^{3} S^1 \times A_i,
\]
each contracting within $T^3$ to $S^1$.
This means that the round category of the three-dimensional torus equals three.

All other three-dimensional manifolds without boundary have Heegaard diagrams of genus greater than one.
They can be presented as gluings of two handlebodies along some homeomorphism of their boundaries.
The boundary of each handlebody is a sphere with a number of handles equal to the Heegaard genus.
Each handlebody can be presented as a union of two closed subsets, each contracting within itself to a circle (Fig.~1.8), that is, of two connected manifolds each containing one full handle $D^1 \times D^2$ and parts of the other handles $D^1 \times D^2$ cut along $x_0 \times D^2$, where $x_0 \in \Int D^1$.

\begin{figure}[h]
\centering\includegraphics[width=0.85\textwidth]{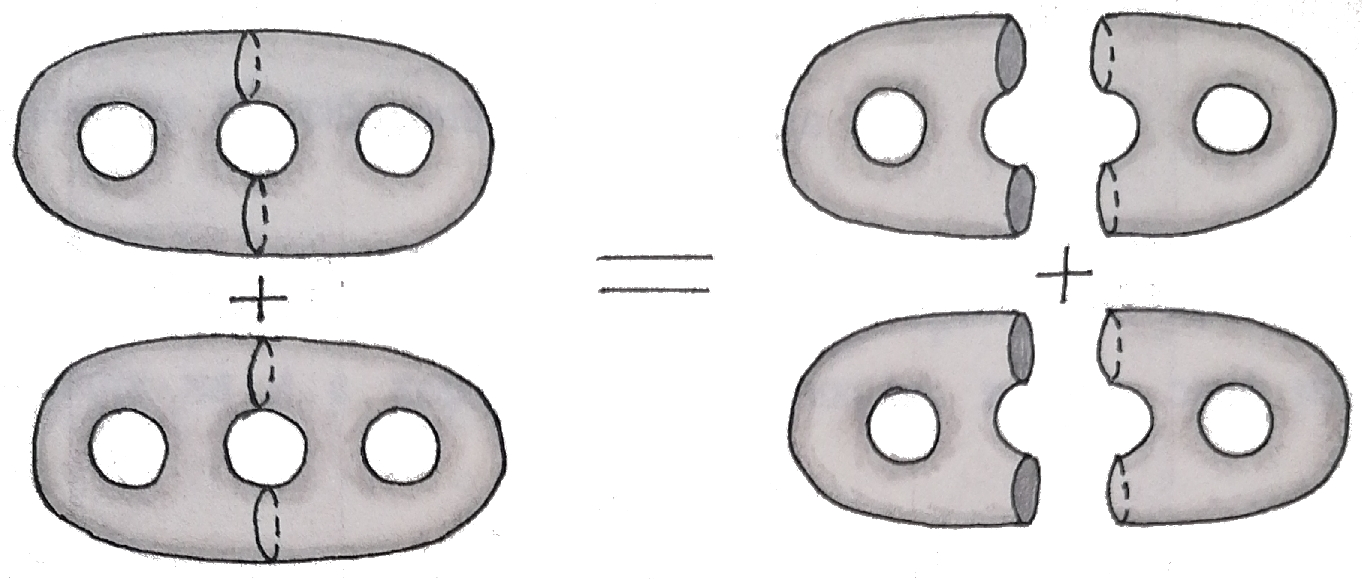}
\caption{}\label{fig:1.8}
\end{figure}

\begin{theorem}[1.13]\label{thm:1.13}
Let
\[
   M = S^{k_1} \times S^{k_2} \times \ldots \times S^{k_r}
\]
be a product of $k_i$-dimensional spheres, $i = 1, \ldots, r$.
Then the Lyusternik--Shnirelman category of $M$ is one more than the number $r$ of factors:
\[
   \cat M = r + 1.
\]
\end{theorem}

\textit{Proof.} By Corollary~\ref{cor:1.8}
\[
   \mathrm{long}\, M = r,
\]
so by Theorem~\ref{thm:1.5}
\[
   \cat M \;\geq\; r + 1.
\]
We show that the Lyusternik--Shnirelman category of $M$ equals exactly $r + 1$.
Realise $M$ as
\begin{align*}
M &= S^{k_1} \times S^{k_2} \times \ldots \times S^{k_r} = \\ &= \{(x_1, \ldots, x_{k_1+1}, y_1, \ldots, y_{k_2+1}, \ldots, z_1, \ldots, z_{k_r+1}) \in \R^l :\\ &\quad x_1^2 + \ldots + x_{k_1+1}^2 = 1,\\ &\quad y_1^2 + \ldots + y_{k_2+1}^2 = 1,\;\ldots,\\ &\quad z_1^2 + \ldots + z_{k_r+1}^2 = 1\},
\end{align*}
where $l = \sum_{i=1}^{r} k_i + r$.

Define on $\R^l$ a differentiable function
\[
   f\colon \R^l \to \R,\quad
   f(x_1,\ldots,x_{k_1+1}, y_1,\ldots,y_{k_2+1},\ldots, z_1,\ldots,z_{k_r+1})
   = x_1 + y_1 + \ldots + z_1.
\]
Its restriction to $M$ is a differentiable function with one maximum
\[
   f_{\max} = r,
\]
one minimum
\[
   f_{\min} = -r
\]
and critical points lying on connected level surfaces
\[
   \{f = -r + 2\},\; \{f = -r + 4\},\; \{f = -r + 6\},\;\ldots,\; \{f = r - 2\}.
\]
The number of these surfaces equals $r - 1$.
The critical points lying on a single level surface can be merged into one (degenerate) critical point \cite[p.\,206]{milnor-morse-1965}.
Thus $f$ has $r + 1$ critical points on $M$, that is, $M$ can be covered by $r + 1$ closed subsets contracting within the manifold.
This is equivalent to the statement of the theorem.
\qed

\begin{theorem}[1.14]\label{thm:1.14}
Let
\[
   M = S^{k_1} \times S^{k_2} \times \ldots \times S^{k_r}
\]
be a product of $k_i$-dimensional spheres, $k_i \geq 1$, $i = 1, \ldots, r$, and let
\[
   P = S^{k_{i_1}} \times S^{k_{i_2}} \times \ldots \times S^{k_{i_l}}
\]
be a submanifold of $M$, where $\{k_{i_1}, k_{i_2}, \ldots, k_{i_l}\} \subset \{k_1, k_2, \ldots, k_r\}$.
Then the $P$-category of $M$ equals
\[
   \Pcat M = r - l + 1.
\]
\end{theorem}

\textit{Proof.} By Theorem~\ref{thm:1.7}
\[
   \mathrm{long}^{\,p} M \;\geq\; r - l.
\]
By Theorem~\ref{thm:1.5}
\[
   \Pcat M \;\geq\; r - l + 1.
\]
We show that the $P$-category of $M$ does not exceed $r - l + 1$.
Express $M$ as a product:
\[
   M = P \times L,
\]
where $P$ is the manifold from the hypothesis and $L$ is the product of the remaining $(r - l)$ spheres.
By Theorem~\ref{thm:1.13} the Lyusternik--Shnirelman category of $L$ is
\[
   \cat L = r - l + 1.
\]
This means $L$ can be covered by $(r - l + 1)$ closed subsets $A_i$
\[
   L = \bigcup_{i=1}^{r-l+1} A_i,
\]
each contracting within $L$ to a point.
Then
\[
   M = P \times L = P \times \bigl(\bigcup_{i=1}^{r-l+1} A_i\bigr)
   = \bigcup_{i=1}^{r-l+1} P \times A_i,
\]
and each subset $P \times A_i$ contracts within $M$ to $P$; that is, $M$ can be covered by $r - l + 1$ closed subsets contracting within $M$ to $P$.
Hence
\[
   \Pcat M \;\leq\; r - l + 1.
\]
Combined with the opposite inequality
\[
   \Pcat M \;\geq\; r - l + 1
\]
this yields the statement of the theorem.
\qed

\begin{corollary}[1.15]\label{cor:1.15}
Let $P$ be the $k$-dimensional torus.
Then the $P$-category of the $n$-dimensional torus equals
\[
   T^k\text{-}\cat T^n = n - k + 1,\quad n \geq k.
\]
\end{corollary}

In particular, the round category of the $n$-dimensional torus equals $n$.

\begin{proposition}[1.16]\label{prop:1.16}
Let $P$ be the $l$-dimensional sphere.
Then the $P$-category of the $n$-dimensional sphere equals:
\begin{enumerate}\setlength{\itemsep}{0pt}
\item[a)] two, if $l < n$;
\item[b)] one, if $l = n$;
\item[c)] zero, if $l > n$.
\end{enumerate}
\end{proposition}

\textit{Proof.}
If $l \neq n$, then $S^l\text{-}\cat S^n \neq 1$ (the homologies of spheres prevent this).
If $l < n$, then from a covering of $S^n$ by two $n$-dimensional discs we construct a covering by two closed subsets contractible within $S^n$ to $S^l$.
Cover $S^n$ by two closed discs $D_1$ and $D_2$ intersecting along the equator $S^n$.
Since the dimension $n$ of the discs exceeds $l$ by at least one, into each disc one can inscribe an $S^l$-sphere meeting the equator in a single point (Fig.~1.9):
\[
   S_1^l \subset D_2^n,\quad S_2^l \subset D_1^n.
\]
Then the sets $S_1^l \cup D_1^n$ and $S_2^l \cup D_2^n$ may be taken as the categorical subsets.

\begin{figure}[h]
\centering\includegraphics[width=6cm]{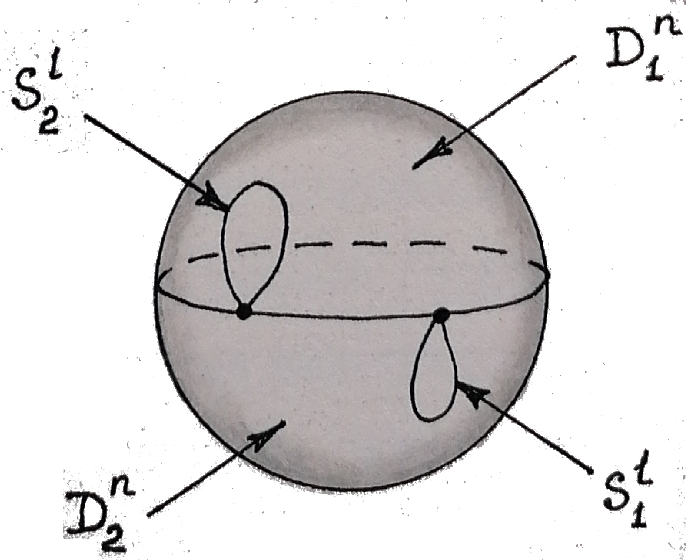}
\caption{}\label{fig:1.9}
\end{figure}

\chapter{Estimate of the number of critical submanifolds of a function on a manifold}

The idea of studying non-degenerate critical submanifolds instead of non-degenerate critical points of a differentiable function on a manifold is due to Bott.
Functions with non-degenerate critical submanifolds are called Morse--Bott functions.
A subclass of them is the so-called round Morse functions, that is, functions whose critical set is a disjoint union of non-degenerate circles.
Such functions were studied in \cite{thurston-1976,franks-1980,franks-1978,sharko-1990,matveev-fomenko-sharko-1988,fomenko-sharko-1989,miyoshi-1983}.

For estimating the number of critical submanifolds of functions with non-degenerate critical submanifolds one may use the Morse inequalities \cite{milnor-morse-1965}, generalised to this setting.

For the wider class of functions with degenerate isolated submanifolds no such estimate of the number of singularities exists in general.

The aim of the present chapter is to estimate the number of degenerate singularities of a function on a manifold when the critical set of the function is a disjoint union of homeomorphic smooth submanifolds, and to give a criterion for the existence of such functions on a manifold.
We also exhibit examples of manifolds for which the lower bound on the number of singularities is sharp.

\section{$P$-functions on a manifold. Conditions of existence}

\begin{definition}[2.1]\label{def:2.1}
A differentiable function on a manifold whose set of critical points is a disjoint union of smooth submanifolds without boundary is called a \emph{$P$-function} if each of its critical submanifolds is homeomorphic to a fixed smooth manifold $P$ without boundary.
\end{definition}

\begin{proposition}[2.2]\label{prop:2.2}
Let $P$ be a smooth submanifold of dimension $p$ in $\R^n$, and let
\[
   f\colon P \times \R^{n-p} \to \R
\]
be a function with the properties:
\begin{itemize}\setlength{\itemsep}{0pt}
\item[a)] $f$ is continuous;
\item[b)] the submanifold $P$ lies in a level surface;
\item[c)] $f\big|_{P \times \R^{n-p}\setminus C}$ is differentiable and has no critical points.
\end{itemize}
Then on $P \times \R^{n-p}$ there exists a function
\[
   F\colon P \times \R^{n-p} \to \R,
\]
having the properties:
\begin{itemize}\setlength{\itemsep}{0pt}
\item[a)] $F$ has a unique critical submanifold $P$;
\item[b)] $F\big|_{P \times \R^{n-p}\setminus C} = f\big|_{P \times \R^{n-p}\setminus C}$ for some compact set $C \subset P \times \R^{n-p}$.
\end{itemize}
\end{proposition}

\textit{Proof.} Let $f(P) = 0$ and let $U$ be a bounded neighbourhood of $P$ in $P \times \R^{n-p}$ of the form
\[
   U = P \times U',
\]
where $U'$ is a bounded open subset of $\R^{n-p}$.
Consider in $U$ a tubular neighbourhood $T_r(P)$ of $P$,
\[
   T_r(P) = \{ y \in U : \rho(y, P) < r \},
\]
where $\rho$ is a Riemannian metric on $P \times \R^{n-p}$.
Denote
\[
   v = T_r(P) \cap f^{-1}(0).
\]
Choose $\varepsilon > 0$ small enough so that the cylindrical neighbourhood $C_\varepsilon(P)$ (see \S2.2) of $P$ is bounded in $P \times \R^{n-p}$.

Consider the retraction
\[
   \pi\colon C_\varepsilon(P) \to v,
\]
which contracts the cylindrical neighbourhood $C_\varepsilon(P)$ onto $v$ along integral trajectories of the vector field $\mathrm{grad}\, f$.

The function $F$ will be obtained by modifying $f$ only on the set $C$, which equals the cylindrical neighbourhood $C_\varepsilon(P)$.
For the construction we use two auxiliary functions
\[
   \lambda\colon \R \to \R \quad\text{and}\quad H\colon v \to [0,1],
\]
having the following properties:
\begin{enumerate}\setlength{\itemsep}{0pt}
\item[(1)] $\lambda f$ is differentiable;
\item[(2)] $\lambda(t) = t$ when $|t| \geq \varepsilon/2$;
\item[(3)] $(d\lambda/dt)_{t_0} > 0$ for $t_0 \neq 0$;
\item[($1'$)] the function $H\pi\colon C_\varepsilon(P) \to [0,1]$ is differentiable and locally flat at every point $x \in P$;
\item[($2'$)] for some open subset $N_0$ with $\overline{N_0} \subset v$, $H(v \setminus N_0) = 1$;
\item[($3'$)] $H(P) = 0$ and $H(v \setminus P) = (0,1]$.
\end{enumerate}

Define $F\colon P \times \R^{n-p} \to \R$ by
\[
   F(x) = \begin{cases}
      f(x), & x \notin C_\varepsilon(P),\\
      (1 - H\pi(x))\cdot \lambda(f(x)) + H\pi(x)\cdot f(x), & x \in C_\varepsilon(P).
   \end{cases}
\]

Properties (1) and ($1'$) ensure that $F$ is differentiable on $C_\varepsilon(P)$.
Outside $C_\varepsilon(P)$ the function $F$ coincides with $f$, since properties (2) and ($2'$) ensure this in a neighbourhood of the boundary $\partial \overline{C_\varepsilon(P)}$.

Hence $F$ is differentiable on its entire domain.
From the definition of $F$ property~(b) holds immediately, taking $C$ to be the closure of $C_\varepsilon(P)$.

Property (a) only needs to be verified for $x \in C_\varepsilon(P)$, since outside $C_\varepsilon(P)$ the function $f$, hence $F$, has no critical points.
On $C_\varepsilon(P)$ we compute the derivative in the direction of $\mathrm{grad}\, f$:
\begin{multline*}
\langle dF, \mathrm{grad}\, f\rangle = \langle d(H\pi(x)), \mathrm{grad}\, f\rangle\cdot\lambda(f(x)) +\\ {} + (1 - H\pi(x))\cdot\langle d(\lambda(f(x))), \mathrm{grad}\, f\rangle +\\ {} + \langle d(H\pi(x)), \mathrm{grad}\, f\rangle\cdot f(x) + H\pi(x)\cdot\langle \mathrm{grad}\, f, df(x)\rangle.
\end{multline*}

The first and third summands vanish for all $x \in C_\varepsilon(P)$, since along integral trajectories of $\mathrm{grad}\, f$ the map $H\pi$ is constant, and on $P$ it is locally flat.

The second summand is positive on $C_\varepsilon(P)$, except on $v \setminus N_0$ and on $v$, where it vanishes.

The last summand is positive on $C_\varepsilon(P)$ except on $P$ and on the integral trajectories entering and leaving $P$, where it vanishes.

Hence the derivative of $F$ in the direction of $\mathrm{grad}\, f$ is positive at every point $x \in C_\varepsilon(P)$ except on $P$, where all summands vanish simultaneously.
This means that $F$ has $P$ as its unique critical submanifold.

To complete the proof we show that one can construct functions $\lambda$ and $H$ with the required properties.

Take $\lambda$ exactly as in~\cite{takens-1968}, where it is constructed as follows.
Let
\[
   \{\lambda_i\colon \R \to \R\}_{i=1}^{\infty}
\]
be a family of differentiable functions with the properties:
\[
   \lambda_i(t) = \begin{cases} t, & \text{if } |t| \geq \varepsilon/2i,\\ 0, & \text{if } |t| \leq \varepsilon/(2(i+1)),\end{cases}
\quad \frac{d}{dt}\bigl[\lambda_i(t)\bigr] \geq 0,\; t \in \R.
\]
The function $\lambda$ is defined as the series
\[
   \lambda = \sum_{i=1}^{\infty} a_i\cdot \lambda_i,
\]
where $\sum_{i=1}^{\infty} a_i = 1$, $a_i = \beta_i\bigl(\sum_{k=1}^{\infty}\beta_k\bigr)^{-1}$, $\beta_i = 2^{-i}/(\gamma_i + 1)$, $\gamma_i = \max\{|\partial_{j_1,\ldots,j_r}(\lambda_i(f))_x|\}$, where the maximum is taken over the set $\{(j_1,\ldots,j_r), x : r \leq i,\; j_1,\ldots,j_r \leq n,\; x \in C_\varepsilon(P)\}$.

To construct $H$, choose a family of nested tubular neighbourhoods $T_{r_i}(P)$, $i = 0, 1, 2, \ldots$, of $P$ in $P \times \R^{n-p}$ with radii $r_i$ satisfying:
\begin{itemize}\setlength{\itemsep}{0pt}
\item[a)] $r > r_0 > r_1 > \ldots$,
\item[b)] $r_i \to 0$ as $i \to \infty$.
\end{itemize}
On each tubular neighbourhood define functions $h_i\colon T_{r_i}(P) \to [0,1]$, $i = 0, 1, 2, \ldots$,
\[
   h_i(y) = \varphi_{r_0 - r_i}(\rho(y, x) - r_i),\quad y \in T_{r_i}(P),\; x \in P,
\]
where $\rho$ is a Riemannian metric on $P \times \R^{n-p}$ and
\[
   \varphi_{r_0 - r_i}(t) = \begin{cases} 0, & \text{if } t \leq 0,\\
   \dfrac{e^{-1/t}}{e^{-1/t} + e^{-1/(r_0 - r_i - t)}}, & \text{if } t > 0.\end{cases}
\]
The functions $\varphi_{r_0 - r_i}(t)$, $i = 0, 1, 2, \ldots$, are: a)~differentiable; b)~$0 \leq \varphi_{r_0 - r_i}(t) \leq 1$; c)~$\varphi_{r_0 - r_i}(t) = 0$ iff $t \leq 0$; d)~$\varphi_{r_0 - r_i}(t) = 1$ for $t \geq r_0 - r_i$ (Fig.~2.1).

\begin{figure}[h]
\centering\includegraphics[width=7cm]{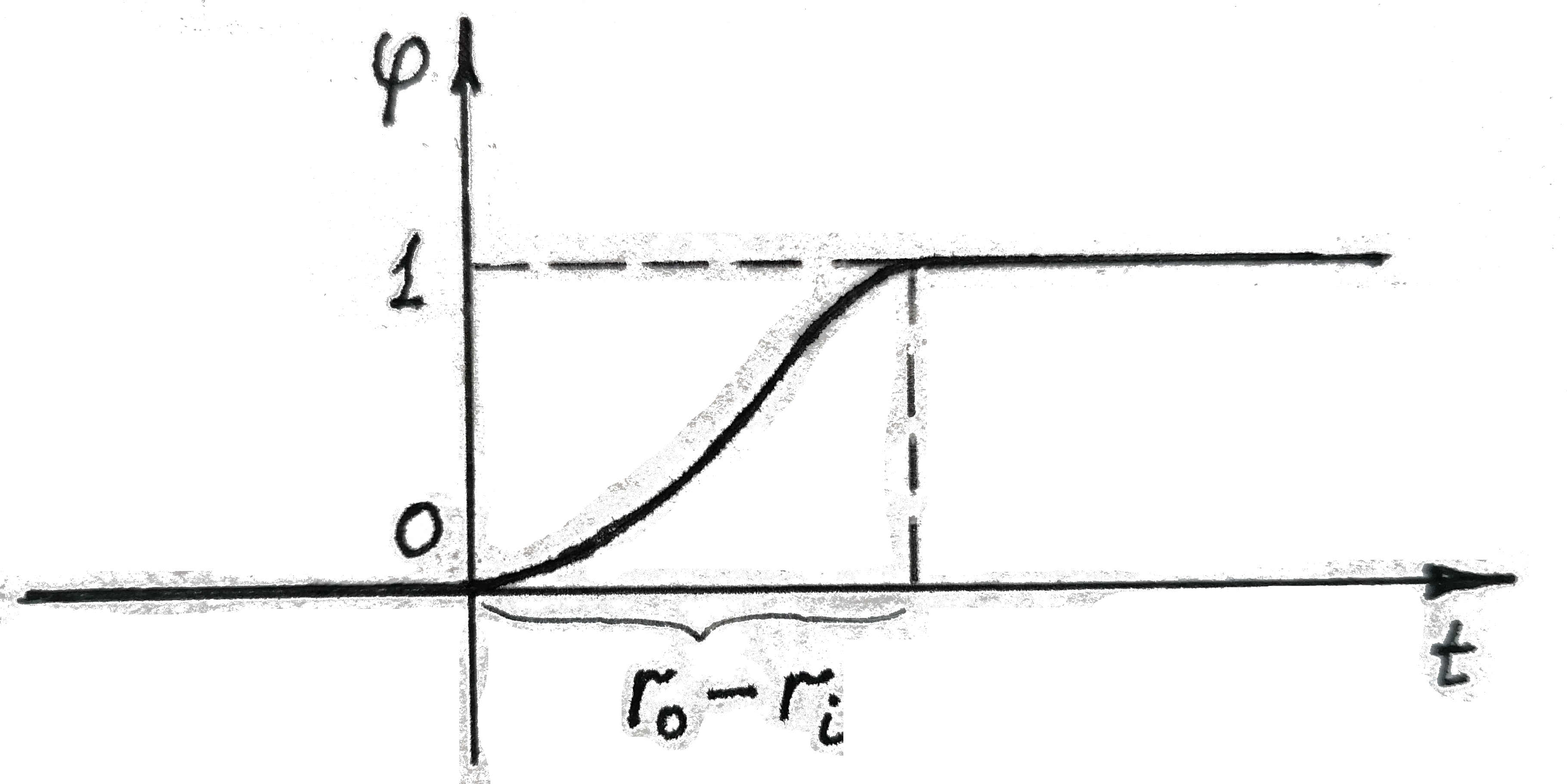}
\caption{}\label{fig:2.1}
\end{figure}

Thus each function $h_i$, $i = 0, 1, 2, \ldots$, has the properties:
\begin{enumerate}\setlength{\itemsep}{0pt}
\item[1)] $h_i$ is differentiable and locally flat at $P$, since $h_i(\overline{T_{r_i}}) = 0$;
\item[2)] $h_i(T_r \setminus T_{r_0}) = 1$;
\item[3)] all the $h_i$ vanish simultaneously only on $P$, since $\bigcap_{i=1}^{\infty} T_{r_i} = P$.
\end{enumerate}

Take the restrictions of $h_i$, $i = 0, 1, 2, \ldots$, to the neighbourhoods
\[
   N_i = T_{r_i} \cap f^{-1}(0)
\]
of $P$ in the level surface $f^{-1}(0)$ and define $H$ by
\[
   H = \sum_{i=0}^{\infty} b_i\cdot h_i\big|_{N_i},
\]
where $b_0, b_1, b_2, \ldots$ are positive numbers with $\sum_{i=0}^{\infty} b_i = 1$.

Verification of properties ($1'$)--($3'$) of $H$, which follow directly from the properties of $h_i$, completes the proof of Proposition~\ref{prop:2.2}.
\qed

\begin{definition}[2.3]\label{def:2.3}\cite{takens-1968}
Let $Q^m$ be a manifold with boundary, and let $S$ be the set of points at which $\partial Q^m$ fails to be smooth.
Assume that every point of $S$ has a neighbourhood in $Q^m$ diffeomorphic to
\[
   \{(x_1,\ldots,x_m) \in \R^m : x_1 \geq 0,\; x_2 \geq 0\}.
\]
$Q^m$ admits a unique smoothing $\widetilde{Q}^m$.
A homeomorphism $\varphi\colon Q^m \to \widetilde{Q}^m$ that is a diffeomorphism $Q^m \setminus S \to \widetilde{Q}^m \setminus \varphi(S)$, as well as its inverse, is called a \emph{special almost diffeomorphism}.
A map obtained as a composition of a diffeomorphism and a special almost diffeomorphism is called an \emph{almost diffeomorphism}.
\end{definition}

\begin{proposition}[2.4]\label{prop:2.4}
Let $M^m$ be a smooth compact manifold with boundary $\partial M^m$, which is a disjoint union of closed subsets $M_1$ and $M_2$, and let $(M^m; M_1, M_2)$ admit a covering $\{Q_i\}_{i=1}^{3}$ by three closed submanifolds satisfying:
\begin{itemize}\setlength{\itemsep}{0pt}
\item[(a)] $\Int(Q_i) \cap \Int(Q_j) = \emptyset$, $i \neq j$;
\item[(b)] $(Q_i, M_i)$ is almost diffeomorphic to $(M_i \times [0,1], M_i \times \{0\})$ for $i = 1, 2$;
\item[(c)] $Q_3$ is almost diffeomorphic to $P \times D^{m-p}$, where $P$ is a smooth submanifold of $M^m$ of dimension $p$;
\item[(d)] $Q_1 \cap Q_2$ is a smooth submanifold of $M^m \setminus \Int Q_3$.
\end{itemize}
Then on $M^m$ there exists a $P$-function $F\colon M^m \to \R$ with the following properties:
\begin{itemize}\setlength{\itemsep}{0pt}
\item $F$ takes a constant minimal value on $M_1$;
\item $F$ takes a constant maximal value on $M_2$;
\item $F$ has no critical points on $M_1$ or $M_2$;
\item $F$ has a unique critical submanifold $P$ lying in $\Int M^m$.
\end{itemize}
\end{proposition}

\begin{remark}
If $M_1 = \emptyset$ or $M_2 = \emptyset$, then $Q_1$ or $Q_2$, respectively, is omitted from the covering $\{Q_i\}_{i=1}^{3}$.
\end{remark}

\textit{Proof.} The fact that $(Q_i, M_i)$ is almost diffeomorphic to $(M_i \times [0,1], M_i \times \{0\})$ for $i = 1, 2$ means that there exists a homeomorphism
\[
   \varphi_i\colon Q_i \to M_i \times [0,1],
\]
which is a diffeomorphism $Q_i \setminus \bigl(\bigcap_{j=1}^{3} Q_j\bigr) \to M_i \times [0,1]\setminus \varphi_i\bigl(\bigcap_{j=1}^{3} Q_j\bigr)$.

Take the differentiable function
\[
   \phi_i\colon M_i \times [0,1] \to \R,\quad i = 1, 2,\quad
   \phi_i(x, t) = 1 - t,
\]
which equals one at $t = 0$.
The composition
\[
   f_i = \phi_i \circ \varphi_i\colon Q_i \to \R,\quad i = 1, 2,
\]
\[
\xymatrix{
   Q_i \ar[rr]^-{\phi_i} \ar[rd]_-{f_i} && 
   M_i \times [0,1] \ar[dl]^-{\psi_i} \\
    & \R
}
\]
is a diffeomorphism
\[
   Q_i \setminus \bigl(\bigcap_{j=1}^{3} Q_j\bigr) \quad\text{onto}\quad [0,1] \setminus f_i\bigl(\bigcap_{j=1}^{3} Q_j\bigr).
\]
On $Q_1 \cup Q_2$ define
\[
   f\colon Q_1 \cup Q_2 \to \R,\quad
   f(x) = \begin{cases} -f_1(x), & x \in Q_1,\\ f_2(x), & x \in Q_2.\end{cases}
\]
It is infinitely differentiable on $Q_1 \cup Q_2 \setminus (Q_1 \cap Q_2)$ and has no critical points there.

We modify $f$ near $Q_1 \cap Q_2$ so that it becomes differentiable on $Q_1 \cup Q_2 \setminus (Q_1 \cap Q_2 \cap Q_3)$ and has no critical points on this set.
For this, take a sufficiently small tubular neighbourhood $T(Q_1 \cap Q_2)$ of the smooth submanifold $Q_1 \cap Q_2$ in $M^m$, diffeomorphic to $(Q_1 \cap Q_2) \times [-1,1]$.

Let $\theta'$ denote this diffeomorphism
\[
   \theta'\colon (Q_1 \cap Q_2)\times[-1,1] \to T(Q_1 \cap Q_2).
\]
It maps the hypersurface $(Q_1 \cap Q_2)\times t_0$ to a hypersurface whose points lie at distance $|t_0|$ from $Q_1 \cap Q_2$.
On $(Q_1 \cap Q_2)\times[-1,1]$ define a further differentiable function
\[
   \theta''\colon (Q_1 \cap Q_2)\times[-1,1] \to \R,\quad
   \theta''(x, t) = t,\quad x \in Q_1 \cap Q_2.
\]
\[
\xymatrix@C=6em{
    T(Q_1 \cap Q_2) \ar[rd]^-{\theta} \\
    (Q_1 \cap Q_2) \times [-1,1] \ar[u]^-{\theta'} \ar[r]_-{\theta''} & \R
}
\]
The composition
\[
   \theta = \theta''\circ (\theta')^{-1}\colon T(Q_1 \cap Q_2) \to \R
\]
is differentiable on $T(Q_1 \cap Q_2)$.

Since $(Q_i, M_i)$ is almost diffeomorphic to $(M_i \times [0,1], M_i \times \{0\})$ for $i = 1, 2$, every point $p \in Q_1 \cap Q_2 \cap Q_3$ admits a chart $\varphi_p^{-1}\colon \R^m \to M^m$ such that $\varphi_p(0,\ldots,0) = p$, $\varphi_p^{-1}(Q_1) = \{(x_1,\ldots,x_m) : x_1 \geq 0,\; x_2 \geq 0\}$, $\varphi_p^{-1}(Q_2) = \{(x_1,\ldots,x_m) : x_1 \geq 0,\; x_2 \leq 0\}$, $\varphi_p^{-1}(Q_3) = \{(x_1,\ldots,x_m) : x_1 \geq 0\}$.

Let $T'(Q_1 \cap Q_2)$ denote the portion of the tubular neighbourhood of $Q_1 \cap Q_2$ containing no points of $\varphi_p^{-1}(Q_1)$ and $\varphi_p^{-1}(Q_2)$ satisfying
\[
   x_1 \geq 0\quad\text{and}\quad x_1 \geq \pm k x_2
\]
for some fixed $k$ (Fig.~2.2).

\begin{figure}[h]
\centering\includegraphics[width=0.85\textwidth]{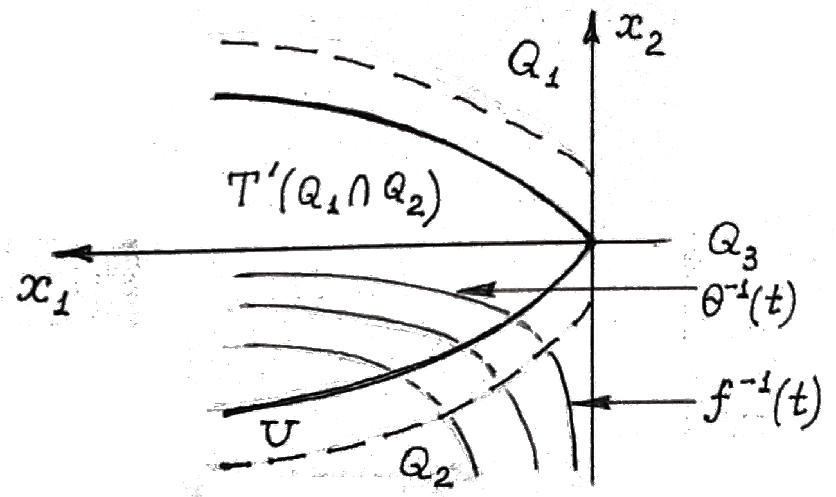}
\caption{}\label{fig:2.2}
\end{figure}

Let $U$ be a sufficiently small neighbourhood of $T'(Q_1 \cap Q_2)$ in $Q_1 \cup Q_2$.
Define
\[
   \bar f\colon Q_1 \cup Q_2 \to \R,\quad
   \bar f(x) = \begin{cases} \theta(x), & x \in T'(Q_1 \cap Q_2),\\ f(x), & x \in Q_1 \cup Q_2 \setminus U,\end{cases}
\]
which on $U \setminus T'(Q_1 \cap Q_2)$ smoothly connects the level surfaces $f^{-1}(t)$ of $f$ with the level surfaces $\theta^{-1}(t)$ of $\theta$.
Extend $\bar f$ by zero on $Q_3$.
Then $\bar f\colon (M^m;\; M_1, M_2) \to \R$ has the following properties:
\begin{enumerate}\setlength{\itemsep}{0pt}
\item[1)] $\bar f$ is continuous;
\item[2)] $\bar f\big|_{M^m \setminus Q_3}$ is differentiable and has no critical points;
\item[3)] $-1 \leq \bar f(M^m) \leq 1$, $\bar f(M_1) = -1$, $\bar f(M_2) = 1$;
\item[4)] $\bar f(Q_3) = 0$.
\end{enumerate}

Since $Q_3$ is almost diffeomorphic to $P \times D^{m-p}$, there exists a continuous map
\[
   \Phi\colon M^m \to M^m,
\]
mapping $Q_3$ to $P$ and acting as a homeomorphism $M^m \setminus Q_3 \to M^m \setminus P$.
$\Phi$ may be taken to be differentiable.
The composition
\[
   \Phi = \bar f \circ \Phi^{-1}\colon M^m \to \R
\]
is a continuous function on $M^m$, differentiable on $M^m \setminus P$, with no critical points outside $P$.
Applying Proposition~\ref{prop:2.2} if necessary to the restriction of $\Phi$ to a small open neighbourhood of $P$ in $M^m$ yields the required function.
\qed

\begin{theorem}[2.5]\label{thm:2.5}
Let $M^m$ be a smooth compact manifold (with or without boundary), and let
\[
   M_1 \subset M_2 \subset \ldots \subset M_k = M^m
\]
be a filtration of $M^m$ by compact manifolds with boundary, satisfying:
\begin{enumerate}\setlength{\itemsep}{0pt}
\item[(a)] $M_i$ are manifolds with boundary, $i = 1, \ldots, k-1$;
\item[(b)] $M_i \subset \Int(M_{i+1})$, $i = 1, \ldots, k-1$;
\item[(c)] $\partial M^m \subset M_k \setminus M_{k-1}$;
\item[(d)] for each $i = 1, \ldots, k$, the manifold
      \[
         (M_i \setminus \Int M_{i-1};\; \partial M_{i-1},\; \partial M_i)
      \]
admits a covering by three closed subsets as in Proposition~\ref{prop:2.4}.
\end{enumerate}
Then on $M^m$ there exists a $P$-function with $k$ critical submanifolds.
\end{theorem}

\textit{Proof.} On each submanifold $(M_i \setminus \Int M_{i-1};\; \partial M_{i-1},\; \partial M_i)$, $i = 1, \ldots, k$, construct a differentiable function
\[
   F_i\colon M_i \setminus \Int M_{i-1} \to \R
\]
as in Proposition~\ref{prop:2.4}, taking the values
\[
   F_i(\partial M_{i-1}) = i - 1,\quad F_i(\partial M_i) = i,
\]
with no critical points on $\partial M_{i-1}$ and $\partial M_i$, and with a single critical submanifold lying in $\Int(M_i \setminus \Int M_{i-1})$.
The function $F\colon M^m \to \R$ defined by
\[
   F(x) = F_i(x),\quad x \in M_i \setminus \Int M_{i-1},
\]
is continuous on $M^m$ and has one critical submanifold $P$ on each $M_i \setminus M_{i-1}$, $i = 1, \ldots, k$.
To make $F$ differentiable on $M^m$, equip $M^m$ with a Riemannian metric $\rho$ and ``correct'' the $F_i$ in sufficiently small neighbourhoods $U(\partial M_i)$ and $U(\partial M_{i-1})$ of the boundaries $\partial M_i$ and $\partial M_{i-1}$ in $M_i \setminus \Int M_{i-1}$ by constructing differentiable functions
\[
   f_i\colon M_i \setminus \Int M_{i-1} \to \R,
\]
\[
   f_i(x) = \begin{cases}
      i - \rho(x, \partial M_i), & x \in U(\partial M_i),\\
      i - 1 + \rho(x, \partial M_{i-1}), & x \in U(\partial M_{i-1}),\\
      F_i(x), & x \in M_i \setminus \Int M_{i-1}\setminus W(\partial M_i)\setminus W(\partial M_{i-1}),
   \end{cases}
\]
where $W(\partial M_i)$ and $W(\partial M_{i-1})$ are sufficiently small neighbourhoods of $\partial M_i$ and $\partial M_{i-1}$ containing $\overline{U(\partial M_i)}$ and $\overline{U(\partial M_{i-1})}$ respectively.
The resulting function $\bar F\colon M^m \to \R$, $\bar F(x) = f_i(x)$, $x \in M_i \setminus M_{i-1}$, is a $P$-function with $k$ critical submanifolds.
\qed

\section{Main theorem}

\begin{definition}[2.6]\label{def:2.6}
Let $P$ be a critical submanifold of a function $f$ on a manifold $M$, and $f(P) = 0$.
We say that an \emph{integral trajectory} $L$ of the vector field $\mathrm{grad}\, f$ \emph{enters} $P$ if
\begin{itemize}\setlength{\itemsep}{0pt}
\item[a)] it is defined for all sufficiently small values of $f$;
\item[b)] every sequence of points $a_1, a_2, a_3, \ldots$ of $L$ for which $f(a_1), f(a_2), f(a_3), \ldots$ decreases to $p$, converges to a point $a \in P$.
\end{itemize}
An integral trajectory $L$ \emph{leaves} a point $a \in P$ if for the function $-f$ it enters the point $a$.
\end{definition}

Let $U$ be a neighbourhood of a critical submanifold $P$ in $M$ whose closure contains no critical points other than those of $P$.

\begin{proposition}[2.7]\label{prop:2.7}
An integral trajectory of $\mathrm{grad}\, f$ passing through an arbitrary point of $U$ either enters $P$, or leaves $P$, or reaches the boundary $\partial \overline{U}$ at some point.
\end{proposition}

\textit{Proof.}
Let $L$ be a segment of an integral trajectory of $\mathrm{grad}\, f$ traversed in the direction of decreasing $f$, passing through an arbitrary point of $U$, and contained in $U$.
Since $f$ is bounded on $U$, there is an infimum $\beta$ of its values on $L$:
\[
   f(x) > \beta,\quad x \in L.
\]
Pick on $L$ a sequence of points
\begin{equation}
a_1, a_2, a_3, \ldots, \tag{1}
\end{equation}
for which $f(a_1), f(a_2), f(a_3), \ldots$ decreases to $\beta$.

We show that sequence~(1) converges.
Assume not.
Then~(1) has at least two limit points $x_1$ and $x_2$.
This means that an arbitrarily small ball-neighbourhood of, say, $x_1$ contains infinitely many points of~(1).
But since a decreasing sequence of ball-neighbourhoods of $x_1$ with radii tending to zero has $x_1$ as its only limit, an arbitrarily small ball-neighbourhood of $x_1$ disjoint from some neighbourhood of $x_2$ must contain almost all points of~(1), so the sequence must converge to $x_1$, contradicting the assumption.

Sequence~(1) cannot converge to a point of $U$ not lying on $P$, since $U$, being open, would then contain points of the trajectory $L$ at which $f$ takes values less than $\beta$.

Hence the limit of~(1) is either a point of $P$ or a point on the boundary of $\overline{U}$.
An analogous statement holds for the segment $L$ traversed in the opposite direction.
\qed

\begin{definition}[2.8]\label{def:2.8}
Let $V$ be a neighbourhood of a critical submanifold $P$ in $M$ such that
\[
   v = V \cap f^{-1}(p) \subset U.
\]
A subset $C_\varepsilon(P)$ of $U$ is called a \emph{cylindrical neighbourhood} of $P$ in $M$ if, for sufficiently small $\varepsilon > 0$, it contains:
\begin{itemize}\setlength{\itemsep}{0pt}
\item[a)] the points $x$ of all integral trajectories of $\mathrm{grad}\, f$ entering $P$ for which $p < f(x) < p + \varepsilon$;
\item[b)] the points $x$ of all integral trajectories of $\mathrm{grad}\, f$ leaving $P$ for which $p - \varepsilon < f(x) < p$;
\item[c)] the points $x$ of all integral trajectories of $\mathrm{grad}\, f$ passing through points of $v$ distinct from $P$ for which $p - \varepsilon < f(x) < p + \varepsilon$;
\item[d)] the submanifold $P$.
\end{itemize}
\end{definition}

\begin{proposition}[2.9]\label{prop:2.9}
There exists $\varepsilon > 0$ small enough that the points described in items a)--d) of Definition~\ref{def:2.8} lie in $U$.
\end{proposition}

\textit{Proof.} Let $W_\eta$ be a neighbourhood of $P$ in $M$ with:
\begin{enumerate}\setlength{\itemsep}{0pt}
\item[1)] $\overline{W_\eta} \subset U$;
\item[2)] $W_\eta = \{ x \in M : \rho(x, v) < \eta \}$, where $\rho$ is a Riemannian metric on $M$ and $v = V \cap f^{-1}(p)$.
\end{enumerate}
We show that one can choose $\varepsilon > 0$ small enough that the points described in items a)--d) lie in $W_\eta$, and hence in $U$.
Suppose otherwise: there exist points of type a)--d) outside $W_\eta$.
By Proposition~\ref{prop:2.7} they must lie on the boundary $\partial\overline{W_\eta}$.

For each $\varepsilon = 1/i$, choose a sequence $q_1, q_2, q_3, \ldots$ of such points, converging to some point $q$ on $\partial\overline{W_\eta}$.
Observe that $q$ is not a critical point.
Since
\[
   |f(q_i)| \leq p + 1/i,
\]
we have $\lim_{i\to\infty} f(q_i) = p$, whence $f(q) = p$.

Choose in $U$ a cylindrical neighbourhood of $q$ disjoint from $v$.
Since it contains almost all $q_i$, it also contains almost all intersection points of the integral trajectories through $q_i$ with the hypersurface $f^{-1}(p)$.
But then these latter cannot lie in $v$, contradicting the fact that $q_i$ are points of type a)--d).
\qed

\begin{corollary}[2.10]\label{cor:2.10}
The tubular neighbourhood of a critical submanifold $P$ in $M$ contains arbitrarily small cylindrical neighbourhoods of $P$.
\end{corollary}

The proof follows from the fact that in $M$ one can find a tubular neighbourhood of $P$ contained in $U$ and containing the closure of $W_\eta$.

\begin{proposition}[2.11]\label{prop:2.11}
The set $C_\varepsilon(P)$ is a neighbourhood of $P$ in $M$, that is, a subset of $M$ containing arbitrarily small tubular neighbourhoods of $P$ in $M$.
\end{proposition}

\textit{Proof.} Choose a small tubular neighbourhood $T$ of $P$ in $M$ such that:
\begin{itemize}\setlength{\itemsep}{0pt}
\item[a)] $|f(x)| < p + \varepsilon$, $x \in T$;
\item[b)] $T \cap f^{-1}(p) \subset C_\varepsilon(P)$.
\end{itemize}
If $C_\varepsilon(P)$ were not a neighbourhood of $P$, there would exist in $T$ a sequence of points $a_1, a_2, a_3, \ldots$ outside $C_\varepsilon(P)$ converging to some $a \in P$.
Passing to a subsequence and replacing $f$ by $-f$ if necessary, we may assume that
\[
   f(a_i) > p,\quad i = 1, 2, 3, \ldots
\]
The integral trajectory through each $a_i$ in the direction of decreasing $f$ cannot enter $P$ or pass through points of $v$, because the $a_i$ lie outside $C_\varepsilon(P)$.
By Proposition~\ref{prop:2.7}, each such trajectory reaches the boundary of $\overline{T}$ at some point $b_i$ with $f(b_i) \geq p$.
Since $\lim_{i\to\infty} a_i = a \in P$,
\[
   \lim_{i\to\infty} f(a_i) = f(a) = p,
\]
hence $\lim_{i\to\infty} f(b_i) = p$.

Let $b$ be a limit point of $b_i$, $i = 1, 2, \ldots$, lying on the boundary of $\overline{T}$.
Then $f(b) = p$.
But every cylindrical neighbourhood of $b$ contains infinitely many $b_i$ as well as infinitely many $a_i$, since
\[
   \lim_{i\to\infty} f(b_i) - \lim_{i\to\infty} f(a_i) = 0,
\]
which contradicts the fact that $a_1, a_2, a_3, \ldots$ converges to $a \in P$.
\qed

\begin{theorem}[2.12]\label{thm:2.12}
Let $M$ be a smooth compact connected manifold (with or without boundary), and let $f$ be a $P$-function on it which, in the case of a manifold with boundary, takes a constant maximal value on the boundary and has no critical points on the boundary.
Then the number of critical submanifolds of $f$ is at least the $P$-category of $M$.
\end{theorem}

\textit{Proof.}
By compactness of $M$, $f$ has finitely many critical submanifolds, hence finitely many critical values.
We argue by induction on the number of critical values.

Consider the minimal value $a$ of $f$.
Let $P_1, \ldots, P_k$ be the corresponding critical submanifolds.
Choose $\varepsilon > 0$ so that $(a, a + \varepsilon)$ contains no critical values.
Then
\[
   M^{a+\varepsilon} = \{ x \in M : f(x) \leq a + \varepsilon \}
\]
is a disjoint union of closures of cylindrical neighbourhoods $\overline{C_\varepsilon(P_i)}$ of $P_1, \ldots, P_k$.
Since $\Pcat \overline{C_\varepsilon(P_i)} = 1$, $i = 1, \ldots, k$,
\[
   \Pcat M^{a+\varepsilon} = \Pcat\bigl(\bigsqcup_{i=1}^{k}\overline{C_\varepsilon(P_i)}\bigr) = k.
\]
Hence the statement holds for the minimal value of $f$.

Assume the statement for some critical value $b$, that is, the number $k$ of critical submanifolds of $f$ on $M^{b+\varepsilon}$ (for sufficiently small $\varepsilon > 0$) satisfies
\[
   k \;\geq\; \Pcat M^{b+\varepsilon}.
\]

Let $c$ be the next critical value, $c > b$, so that no critical values lie in $(b, c)$.
Let $P_1, \ldots, P_s$ be the critical submanifolds for value $c$.
Choose $\delta > 0$ so that $(c - \delta, c + \delta)$ contains no critical values other than $c$ ($\delta$ may be shrunk if needed).

Since the critical manifolds in $M$ are smooth submanifolds without boundary, each has a tubular neighbourhood in $M$.
Choose in $M$ tubular neighbourhoods $T(P_i)$, $i = 1, \ldots, s$, of the critical submanifolds small enough that
\begin{itemize}\setlength{\itemsep}{0pt}
\item[a)] they lie in the critical layer $f^{-1}(c - \delta, c + \delta) = \{ x \in M : c - \delta < f(x) < c + \delta \}$;
\item[b)] $\overline{T(P_i)}\cap\overline{T(P_j)} = \emptyset$, $i \neq j$;
\item[c)] the sets $v_i = T(P_i) \cap f^{-1}(c)$ are small enough that the closures of the cylindrical neighbourhoods $\overline{C_\delta(P_i)}$ are pairwise disjoint.
\end{itemize}
Present $M^{c+\delta}$ as
\[
   M^{c+\delta} = M^{b+\varepsilon} \cup f^{-1}[b+\varepsilon, c - \delta] \cup f^{-1}[c - \delta, c + \delta].
\]
After removing from the critical layer $f^{-1}[c - \delta, c + \delta]$ all the cylindrical neighbourhoods $\overline{C_\delta(P_i)}$, $i = 1, \ldots, s$, the closure of the remaining set
\[
   M^{c+\delta} \setminus \bigcup_{i=1}^{s} \overline{C_\delta(P_i)}
   = M^{b+\varepsilon} \cup f^{-1}[b+\varepsilon, c - \delta] \cup
      \bigl(f^{-1}[c - \delta, c + \delta] \setminus \bigcup_{i=1}^{s} \overline{C_\delta(P_i)}\bigr)
\]
can be continuously deformed onto $M^{b+\varepsilon}$: first along integral trajectories of $\mathrm{grad}\, f$ onto
\[
   M^{b+\varepsilon} \cup f^{-1}[b+\varepsilon, c - \delta],
\]
then again along integral trajectories onto $M^{b+\varepsilon}$.
Hence $M^{b+\varepsilon}$ is a deformation retract of $\overline{M^{c+\delta} \setminus \bigcup_{i=1}^{s} \overline{C_\delta(P_i)}}$.

For each critical submanifold $P_i$, $i = 1, \ldots, s$,
\[
   \Pcat_{M^{c+\delta}}\overline{C_\delta(P_i)} = 1,
\]
since the closure of each cylindrical neighbourhood $\overline{C_\delta(P_i)}$ can be contracted onto $\bar v_i$, which in turn contracts onto $P_i$ by the deformation $T(P_i) \to P_i$.

By the properties of $P$-category,
\begin{align*}
\Pcat M^{c+\delta} &= \Pcat_{M^{c+\delta}}\bigl[\bigl(\overline{M^{c+\delta} \setminus \bigsqcup_{i=1}^{s}\overline{C_\delta(P_i)}}\bigr)\cup\bigl(\bigsqcup_{i=1}^{s}\overline{C_\delta(P_i)}\bigr)\bigr]\\ &\leq \Pcat_{M^{c+\delta}} M^{b+\varepsilon} + \Pcat_{M^{c+\delta}}\bigsqcup_{i=1}^{s}\overline{C_\delta(P_i)}\\ &\leq \Pcat M^{b+\varepsilon} + s,
\end{align*}
whence, using the inductive hypothesis,
\[
   \Pcat M^{c+\delta} \;\leq\; k + s,
\]
which completes the proof.
\qed

\section{Exact $P$-functions on manifolds}

\begin{definition}[2.13]\label{def:2.13}
A $P$-function is called \emph{exact} on a manifold if the number of its critical submanifolds equals the minimum, taken over all $P$-functions on the manifold, of the number of critical submanifolds.
\end{definition}

Almost all functions in Chapter~III will be exact.
Here we give some examples of the existence and construction of exact $P$-functions on manifolds.

\begin{definition}[2.14]\label{def:2.14}
A $P$-function on a manifold is called \emph{round} if its critical submanifolds are homeomorphic to the circle $S^1$.
\end{definition}

\begin{theorem}[2.15]\label{thm:2.15}
Let $\{k_{i_1}, \ldots, k_{i_l}\}$ be a subset of $\{k_1, \ldots, k_n\}$ positive integers.
Let
\[
   P = S^{k_{i_1}} \times \ldots \times S^{k_{i_l}}
\]
be a product of $k_{i_j}$-dimensional spheres, $j = 1, \ldots, l$.
Then on the manifold
\[
   M = S^{k_1} \times \ldots \times S^{k_n}
\]
there exists an exact $P$-function with $n - l + 1$ singularities.
\end{theorem}

\textit{Proof.}
By Theorem~\ref{thm:1.14}, the number of singularities of a $P$-function on $M$ is at least $n - l + 1$.
We show that on $M$ there exists a $P$-function with exactly $n - l + 1$ singularities.
Represent $M$ as
\[
   M = P \times L,
\]
where $P$ is the given submanifold and $L$ is the product of the remaining factors of $M$.
For notational convenience set
\[
   P = S^{k_1} \times \ldots \times S^{k_l},
\]
so that
\[
   L = S^{k_{l+1}} \times \ldots \times S^{k_n}.
\]
Write
\begin{multline*}
L = S^{k_{l+1}} \times \ldots \times S^{k_n} =\\ = \{(x_1,\ldots,x_{k_{l+1}+1},\ldots,y_1,\ldots,y_{k_n+1}) \in \R^m :\\ x_1^2 + \ldots + x_{k_{l+1}+1}^2 = 1,\;\ldots,\; y_1^2 + \ldots + y_{k_n+1}^2 = 1\},
\end{multline*}
where $m = \sum_{i=l+1}^{n} k_i + n - l$.

Define on $P \times \R^m$ the differentiable function
\[
   f\colon P \times \R^m \to \R,\quad
   f(P, x_1, \ldots, y_{k_n+1}) = x_1 + \ldots + y_1.
\]
Its restriction to $L$ is a differentiable function with one minimum, one maximum, and critical points lying on $(n - l - 1)$ connected level surfaces.
Critical points on a single level surface can be merged into one (degenerate) critical point \cite[p.\,206]{milnor-morse-1965}.
The resulting function has $n - l + 1$ critical points.
Its restriction to $M$ is the desired $P$-function with the required number of singularities.
\qed

\begin{corollary}[2.16]\label{cor:2.16}
Let $P$ be the $k$-dimensional torus, $k \leq n$.
Then on the $n$-dimensional torus there exists an exact $P$-function with $n - k + 1$ singularities.
In particular, an exact round function on the $n$-dimensional torus has $n$ singularities.
\end{corollary}

\begin{theorem}[2.17]\label{thm:2.17}
Let $P = S^n$ be the $n$-dimensional sphere.
Then on the odd-dimensional sphere
\[
   S^{2n+1},\quad n \geq 1,
\]
there exists an exact $P$-function with two singularities.
\end{theorem}

\textit{Proof.} Realize $S^{2n+1}$ as the quotient of the $(2n+1)$-dimensional disc by its boundary:
\[
   S^{2n+1} = D^{2n+1}/S^{2n}.
\]
The disc
\[
   D^{2n+1} = \{(x_1,\ldots,x_{2n+1}) \in \R^{2n+1} : x_i^2 \leq 1,\; i = 1, \ldots, 2n+1\}
\]
is the union of three subsets:
\begin{align*}
N_1 &= \{(x_1,\ldots,x_{2n+1}) \in \R^{2n+1} : 1 - \varepsilon \leq x_i^2 \leq 1,\; i = 1, \ldots, 2n+1\},\\ N_2 &= \{(x_1,\ldots,x_{2n+1}) \in \R^{2n+1} : x_i^2 \leq 1 - \varepsilon,\; i = 1, \ldots, n;\; 0 \leq x_j^2 \leq \varepsilon,\; j = n+1, \ldots, 2n+1\},\\ N_3 &= \overline{D^{2n+1} \setminus (N_1 \cup N_2)}\\ &= \{(x_1,\ldots,x_{2n+1}) \in \R^{2n+1} : x_i^2 \leq 1 - \varepsilon,\; i = 1, \ldots, n;\; \varepsilon \leq x_j^2 \leq 1 - \varepsilon,\; j = n+1, \ldots, 2n+1\}
\end{align*}
for sufficiently small $\varepsilon > 0$.

After the quotient $S^{2n+1} = D^{2n+1}/S^{2n}$, consider the covering of $S^{2n+1}$ by two subsets $M_1$ and $M_2$, the images of $N_1 \cup N_2$ and $N_3$, respectively.

It is not difficult to see that each of $M_1$ and $M_2$ is almost diffeomorphic to
\[
   S^n \times D^{n+1}.
\]
Hence the statement follows from Theorem~\ref{thm:2.5}.
\qed

\chapter{$P$-functions on manifolds}

It is known that on every two-dimensional smooth compact connected manifold without boundary there exists a differentiable function with three critical points.
In most cases such a function is given by means of its smooth level lines.
Does there exist an analytic way of presenting such a function?
Can one construct a round function on a three-dimensional manifold, that is, a function with isolated critical circles?
Does there exist on a given manifold a function with prescribed singularities?
Complete or partial answers to these and other questions are given in this chapter.
Some of the results presented here can also be obtained by other means.
For example, the existence of round functions on three-dimensional manifolds can be derived from the works of Wilson, Frank and Franks~\cite{wilson-1966,frank-1988,franks-1980}.

This chapter is devoted mainly to the construction of functions with degenerate singularities on certain manifolds.
It offers interesting material for further investigations not addressed in the present work.

\section{$P$-functions on two-dimensional manifolds}

On two-dimensional smooth compact connected manifolds we consider $P$-functions with
\begin{enumerate}\setlength{\itemsep}{0pt}
\item[1)] $P$ a point,
\item[2)] $P$ the circle $S^1$.
\end{enumerate}

In the first case, a $P$-function is a function with isolated critical points, in general degenerate.
It is known that the minimal number of isolated critical points on the sphere $S^2$ is two, and on the remaining two-dimensional manifolds without boundary it is three.
A function with three critical points on a manifold without boundary is, in most cases, constructed by means of its level lines.

Let us present another way of constructing a function with three critical points on a manifold $M$ without boundary that is not homeomorphic to $S^2$.
For this, represent $M$ as a covering by three discs $D_1$, $D_2$, $D_3$ (Fig.~3.1) with the corresponding filtration
\[
   M_1 \subset M_2 \subset M_3 = M,
\]
where $M_1 = D_1$, $M_2 = D_1 \cup D_2$.

After appropriate smoothing of corners, the manifolds $M_i$, $i = 1, 2, 3$, satisfy the hypotheses of Theorem~\ref{thm:2.5}, whence the existence on $M$ of a function with three critical points follows.

This function is exact, since the Lyusternik--Shnirelman category of the manifold equals three.

\begin{figure}[h]
\centering\includegraphics[width=0.85\textwidth]{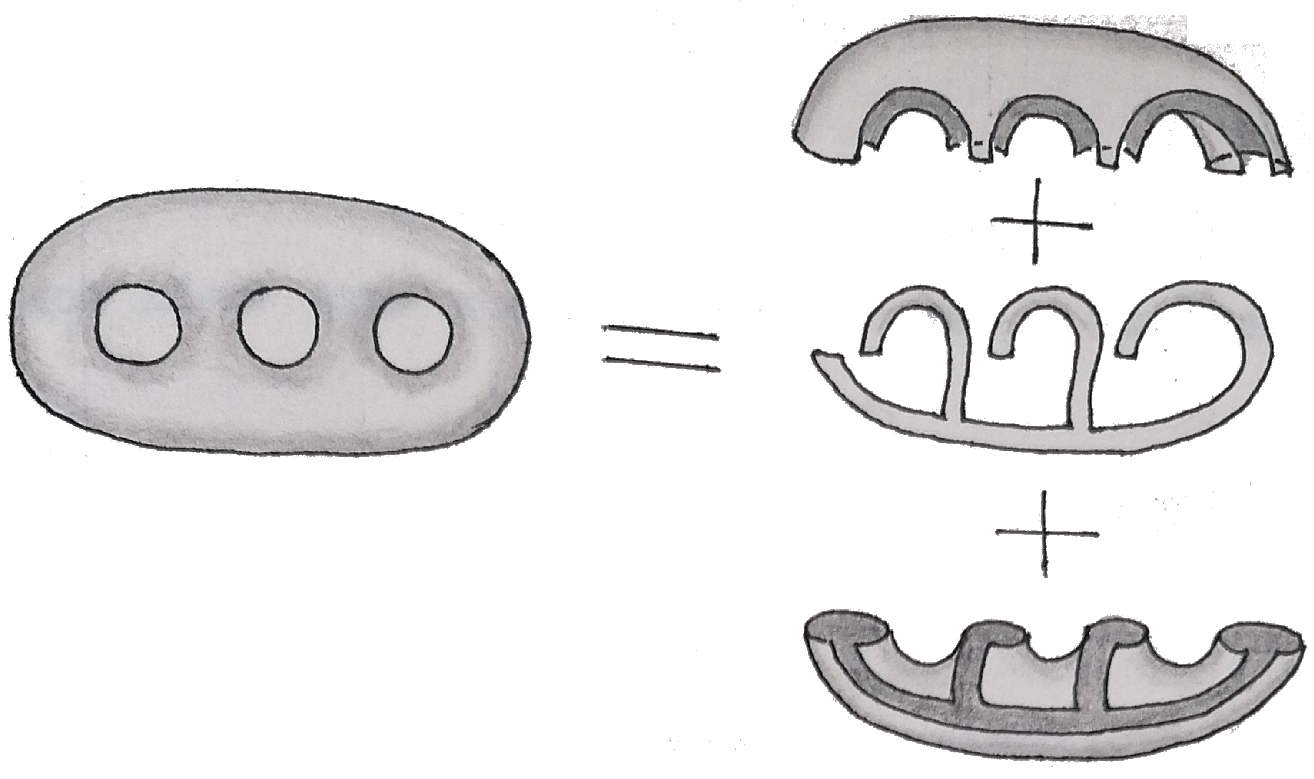}
\caption{}\label{fig:3.1}
\end{figure}

\begin{proposition}[3.1]\label{prop:3.1}
On a two-dimensional smooth compact connected manifold with boundary, not homeomorphic to the disc $D^2$, there exists an exact function with two critical points.
\end{proposition}

\textit{Proof.} Any two-dimensional compact connected manifold with boundary admits a filtration by two manifolds with boundary
\[
   M_1 \subset M_2 = M,
\]
satisfying the hypotheses of Theorem~\ref{thm:2.5}.
This follows from the covering of the manifold by two discs (Fig.~1.3).
By Proposition~\ref{prop:1.10} and Theorem~\ref{thm:2.12}, the constructed function is exact.
\qed

We now consider the second type of $P$-functions --- round functions.

\begin{theorem}[3.2]\label{thm:3.2}
On a manifold homeomorphic to the sphere $S^2$, to the projective plane $\R P^2$, or to the disc $D^2$, there exist no round functions.
On every other two-dimensional smooth compact connected manifold round functions exist.
\end{theorem}

\textit{Proof.}
Suppose a round function $f$ exists on the sphere $S^2$.
By compactness of $S^2$, the number of its critical circles is finite.
Take any critical value $c$ of $f$ and a neighbourhood $(c - \varepsilon, c + \varepsilon)$ for some $\varepsilon > 0$ such that no other critical value lies in $[c - \varepsilon, c + \varepsilon]$.
The critical value $c$ corresponds to a finite number of critical circles.
Remove from $S^2$ the critical layer
\[
   f^{-1}[c - \varepsilon, c + \varepsilon] = M_c.
\]
By Jordan's theorem the closure of the remaining set is a disjoint union of submanifolds $D_1$ and $D_2$:
\[
   \overline{S^2 \setminus f^{-1}[c - \varepsilon, c + \varepsilon]} = D_1 \cup D_2,
\]
each with a single boundary component
\[
   \partial D_1 = f^{-1}(c - \varepsilon)\quad\text{and}\quad
   \partial D_2 = f^{-1}(c + \varepsilon).
\]
Suppose that in one of these submanifolds, say $D_1$, there are no critical points of $f$.
Then $D_1$ must be homeomorphic to
\[
   \partial D_1 \times I,
\]
that is, to a manifold with two boundary components.
This contradiction shows that $D_1$ and $D_2$ must contain critical circles.

Take the next critical value $b$ of $f$ (say $b > c$) such that $(c, b)$ contains no critical values, and $\delta > 0$ small enough that $(b, b + \delta)$ also contains no critical values.
The value $b$ corresponds to finitely many critical circles.
Removing the layer $f^{-1}[c + \varepsilon, b + \delta] = M_b$ from $D_2$, we obtain
\[
   D_3 = \overline{D_2 \setminus M_b}.
\]
Since $D_3$ has only one boundary component $\partial D_3 = f^{-1}(b + \delta)$, $D_3$ must also contain critical circles.
Continuing in this way to remove critical layers we exhaust the set of critical values.
But the remaining submanifold will have a single boundary component and therefore must contain critical points.
This contradiction shows that no round function exists on a manifold homeomorphic to $S^2$.
Analogous arguments for manifolds homeomorphic to $\R P^2$ or to $D^2$ show that no round functions exist on these manifolds either.

On every two-dimensional smooth compact connected manifold without boundary, other than $S^2$ and $\R P^2$, round functions exist.
This is visually clear from the example of the height function.
Figure~3.2 displays a manifold homeomorphic to a sphere with three handles, on which the height function has six critical circles.

\begin{figure}[h]
\centering\includegraphics[width=0.85\textwidth]{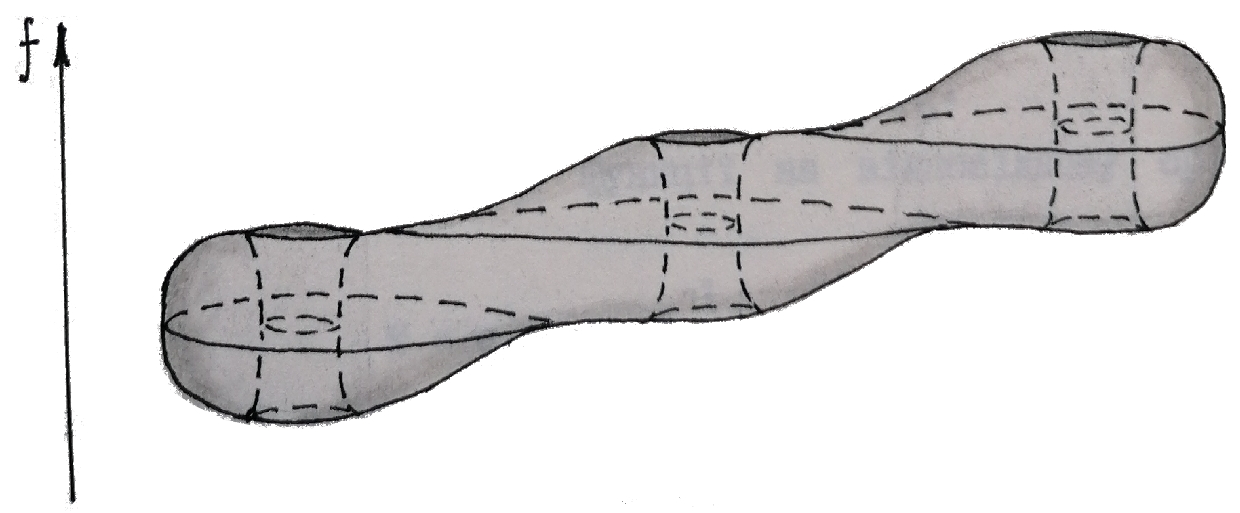}
\caption{}\label{fig:3.2}
\end{figure}

A round function on a manifold without boundary can also be defined by its level lines.
For this, represent the manifold as a fundamental polygon $W$ in the symmetric canonical form
\[
   W = a_1\ldots a_n a_1^{-1}\ldots a_{n-1}^{-1} a_n^{\,e},
\]
where $e = -1$ corresponds to the orientable case and $e = 1$ to the non-orientable one, and transfer the level lines (say, of the height function) onto the fundamental polygon.
Figure~3.3 shows level lines of a round function on the double torus.

\begin{figure}[h]
\centering\includegraphics[width=0.85\textwidth]{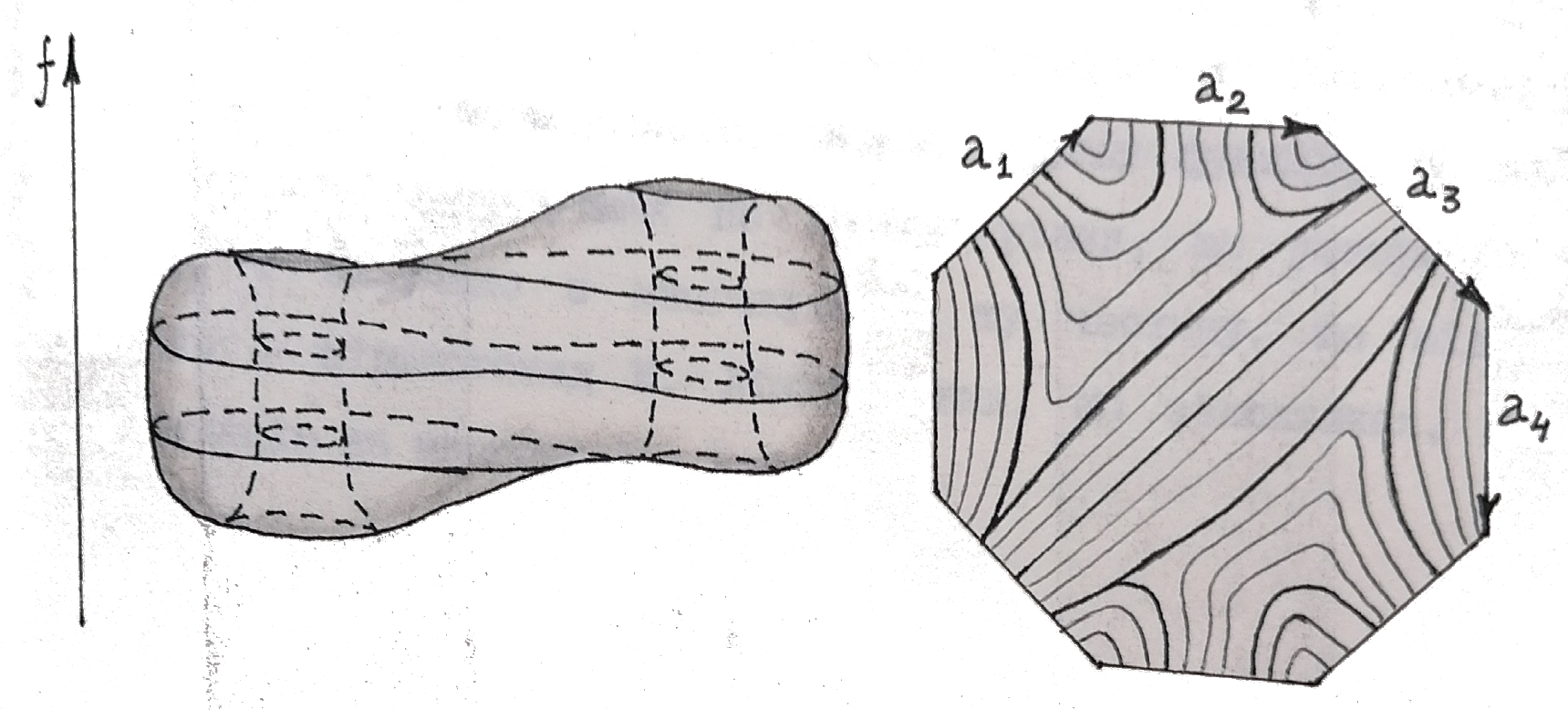}
\caption{}\label{fig:3.3}
\end{figure}

On a non-orientable manifold a round function can also be given by its level lines.
Suppose the non-orientable manifold has the symmetric canonical form
\[
   W = a_1\ldots a_n a_1^{-1}\ldots a_{n-1}^{-1} a_n.
\]
Define the level lines of a round function on the corresponding orientable manifold
\[
   W = a_1\ldots a_n a_1^{-1}\ldots a_{n-1}^{-1} a_n^{-1}
\]
so that the side $a_n$ of the fundamental polygon corresponds to one of the critical circles on the manifold.
These level lines will also be level lines of a round function on the corresponding non-orientable manifold.

On a manifold with boundary we define a round function via its level lines as follows.
Recall that on manifolds with boundary we consider only functions taking a constant maximal value on the boundary and having no critical points there.
Adjoin to $M$ (with boundary $\partial M$) a manifold $N$ with boundary homeomorphic to $\partial M$, so that the ``glued'' manifold has no boundary.
One may take $N = M$ (Fig.~3.4).

On the fundamental polygon of the ``glued'' manifold without boundary, define the round function via its level lines as above.
Cut the polygon along the level lines forming the boundary of $M$ and remove from the polygon the part corresponding to the adjoined piece $N$.
The remaining level lines give a round function on the manifold with boundary.

\begin{figure}[h]
\centering\includegraphics[width=0.85\textwidth]{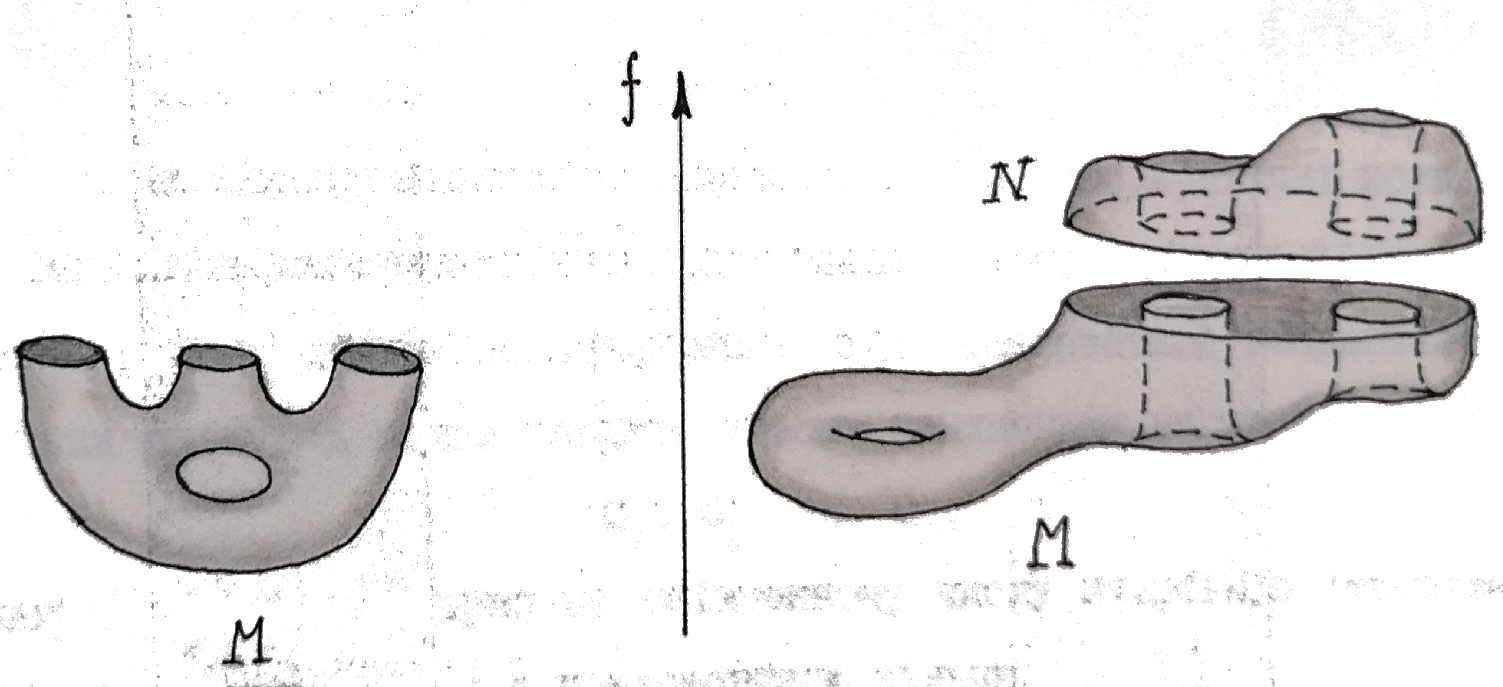}
\caption{}\label{fig:3.4}
\end{figure}

\begin{theorem}[3.3]\label{thm:3.3}
On every two-dimensional smooth compact connected manifold not homeomorphic to the disc $D^2$, the sphere $S^2$ or the projective plane $\R P^2$, there exists an exact round function whose number of singularities equals:
\begin{enumerate}\setlength{\itemsep}{0pt}
\item two on the torus and the Klein bottle;
\item three on the remaining manifolds without boundary;
\item one on the annulus $S^1 \times I$ and on the Möbius band;
\item two on the remaining manifolds with boundary.
\end{enumerate}
\end{theorem}

\textit{Proof.} Let the torus be given by
\[
   T^2 = S^1 \times S^1 = \{(x_1, \ldots, x_4) \in \R^4 :\;
      x_1^2 + x_2^2 = 1,\; x_3^2 + x_4^2 = 1\}.
\]
The restriction of $f = x_1$ to $S^1 \times S^1$ is an exact function with two critical circles.

The Klein bottle is obtained from the torus by cutting along the circle
\[
   \{(x_1, \ldots, x_4) \in \R^4 :\; x_1 = 0,\; x_2 = 1,\; x_3^2 + x_4^2 = 1\}
\]
and identifying the boundaries via
\[
   (0, 1, x_3, x_4) \mapsto (0, 1, x_3, -x_4).
\]
The function $f = x_1$, restricted to the Klein bottle, is exact and has two critical circles.

On the remaining manifolds without boundary an exact function is defined by its smooth level lines (the function is naturally not uniquely determined).
Since such a manifold can be covered by three closed sets $Q_1$, $Q_2$, $Q_3$ intersecting pairwise along finitely many arcs (Proposition~\ref{prop:1.11}), each homeomorphic to the annulus $S^1 \times I$, take as level surfaces in $Q_1$ and $Q_2$ the circles
\[
   S^1 \times t_0,\quad\text{where } t_0 \in I.
\]
Extending in $Q_3$ neighbourhoods of the intersection arcs $Q_1 \cap Q_3$ and $Q_2 \cap Q_3$ so that those arcs lie on the circle in $Q_3$
\[
   S^1 \times t_0,\quad t_0 = 0.5,
\]
we obtain a system of level lines for a function on the manifold.

On the annulus $S^1 \times I$ given by
\[
   \{(x_1, x_2, x_3) \in \R^3 :\; x_1^2 + x_2^2 = 1,\; 0 \leq x_3 \leq 1\},
\]
the function
\[
   f = (x_3 - 0.5)^2
\]
with one critical circle is exact.

The Möbius band is obtained from $S^1 \times I$ by cutting along
\[
   \{(x_1, x_2, x_3) \in \R^3 :\; x_1 = 0,\; x_2 = 1,\; 0 \leq x_3 \leq 1\}
\]
and identifying the boundaries via
\[
   (0, 1, x_3) \mapsto (0, 1, -x_3).
\]
The function
\[
   f = (x_3 - 0.5)^2
\]
with one critical circle is exact.

On the remaining manifolds with boundary one may define a round function via its smooth level lines analogously to what was done for manifolds without boundary.
\qed

\section{$P$-functions on three-dimensional manifolds}

\begin{theorem}[3.4]\label{thm:3.4}
On every three-dimensional smooth compact connected manifold without boundary a round function exists.
If the manifold is homeomorphic to one of
\begin{enumerate}\setlength{\itemsep}{0pt}
\item[a)] the sphere $S^3$,
\item[b)] the projective space $\R P^3$,
\item[c)] the manifold $S^1 \times S^2$,
\item[d)] a lens space,
\end{enumerate}
then it admits an exact round function with two critical circles.
If the manifold is homeomorphic to the torus $T^3$, then it admits an exact round function with three critical circles.
On every other three-dimensional smooth compact connected manifold there exists a round function whose number of singularities does not exceed four.
\end{theorem}

\textit{Proof.}
It is known that if a three-dimensional manifold is homeomorphic to one of $S^3$, $\R P^3$, $S^1 \times S^2$, or a lens space, then it admits a Heegaard diagram of genus one.
This means that each such manifold can be represented as the gluing of two solid tori
\[
   S^1 \times D^2
\]
along some homeomorphism of their boundaries.
If on one of the solid tori one defines a differentiable function taking a constant maximal value on the boundary $S^1 \times S^1$ and having one critical circle, then the required round function is obtained by extending it to the second solid torus by the negative of the function plus an appropriate constant.
The resulting function is exact.
This follows from the main theorem and Proposition~\ref{prop:1.12}.

If the three-dimensional manifold is homeomorphic to the torus $T^3$, then by Corollary~\ref{cor:2.16} there exists on it an exact round function with three critical circles.

Every other three-dimensional smooth compact connected manifold $M$ without boundary admits a Heegaard diagram of genus $g \geq 2$.
This means that $M$ may be presented as the gluing of two handlebodies $M_1$ and $M_2$ with $g$ handles each, where each handle
\[
   D^2 \times I
\]
is attached to the ball along the set
\[
   D^2 \times S^0.
\]
Cut each of the handlebodies along $g$ handles into two parts so that each is homeomorphic to
\[
   S^1 \times D^2
\]
(Fig.~1.8).
Denote these parts by $N_i$, $i = 1, \ldots, 4$, so that
\[
   M_1 = N_1 \cup N_2,\quad M_2 = N_3 \cup N_4.
\]
Define on $M_1$ a differentiable function $f$ taking a constant maximal value $a$ on $\partial M_1$ and having two critical circles.
This can be done, for instance, as in Theorem~\ref{thm:2.5}, by taking an appropriate filtration of $M_1$.

The function $F\colon M \to \R$ defined by
\[
   F = \begin{cases} f(x), & \text{if } x \in M_1,\\
      2a - f(x), & \text{if } x \in M_2,\end{cases}
\]
is the required one.
\qed

\section{$P$-functions on $n$-dimensional manifolds}

\begin{theorem}[3.5]\label{thm:3.5}
Let $P = S^{n-1}$ be the $(n-1)$-dimensional sphere.
Then on the $n$-dimensional sphere there are no $P$-functions.
\end{theorem}

\textit{Proof.}
Suppose an $S^{n-1}$-function $f$ exists on $S^n$.
Take any of its values $a$ and the corresponding critical submanifold --- a sphere $S_1^{n-1}$.

By the Jordan--Brouwer theorem, every $(n-1)$-dimensional submanifold of $\R^n$ homeomorphic to a sphere divides the ambient space into two components and is their common boundary.
If from $S^n$ we remove a sufficiently small cylindrical neighbourhood $C_\varepsilon(S_1^{n-1})$ of the critical sphere $S_1^{n-1}$, we obtain a disjoint union of two connected components $D_1$ and $D_2$ with boundaries
\[
   \partial D_1 = f^{-1}(a - \varepsilon)\quad\text{and}\quad
   \partial D_2 = f^{-1}(a + \varepsilon)
\]
respectively (for sufficiently small $\varepsilon > 0$).

The submanifold $D_1$ must contain critical points; otherwise it would be homeomorphic to
\[
   \partial D_1 \times I,
\]
that is, to a manifold with two boundary components.
Let the critical sphere $S_2^{n-1}$ in $D_1$ correspond to a critical value $b < a$ such that $(b, a)$ contains no critical values.
Removing from $D_1$ the subset
\[
   f^{-1}(a - \varepsilon, b + \delta)\cup C_\varepsilon(S_2^{n-1})
\]
for small $\delta > 0$, we obtain a submanifold $D_1'$ with single boundary component
\[
   \partial D_1' = f^{-1}(b - \delta),
\]
which must again contain critical points.
Continuing in this way to remove critical layers we exhaust the supply of critical submanifolds, which is finite by compactness of $S^n$.
A submanifold with single boundary component remains, in which there must be no critical points.
This contradiction completes the proof.
\qed

\begin{theorem}[3.6]\label{thm:3.6}
Let $P = S^1$ be the one-dimensional sphere, and let $n \geq 3$.
Then on the $n$-dimensional sphere there exists a $P$-function whose number of singularities does not exceed
\[
   [n/2] + 1.
\]
\end{theorem}

\textit{Proof.} By Theorem~\ref{thm:3.4}, on the three-dimensional sphere there exists a round function with two singularities.

Consider the four-dimensional sphere $S^4$.
Realize it as the quotient
\[
   S^4 = D^4 / S^3.
\]
Let the disc $D^4$ be given by
\[
   D^4 = \{(x_1, \ldots, x_4) \in \R^4 :\; x_i^2 \leq 1,\; i = 1, \ldots, 4\}.
\]
Consider three subsets of the disc:
\begin{align*}
N_1 &= \{(x_1, \ldots, x_4) \in \R^4 :\; x_1^2 \leq 1,\; x_i^2 \leq \varepsilon,\; i = 2, 3, 4\},\\ N_2 &= \{(x_1, \ldots, x_4) \in \R^4 :\; 1 - \varepsilon \leq x_i^2 \leq 1,\; i = 1, \ldots, 4\},\\ N_3 &= \overline{D^4 \setminus (N_1 \cup N_2)}\\ &= \{(x_1, \ldots, x_4) \in \R^4 :\; x_1^2 \leq 1 - \varepsilon,\; \varepsilon \leq x_i^2 \leq 1 - \varepsilon,\; i = 2, 3, 4\}
\end{align*}
for sufficiently small $\varepsilon > 0$.
Cover $S^4$ by two subsets $M_1$ and $M_2$, the images of $N_1 \cup N_2$ and $N_3$ under the quotient.
$M_1$ is homeomorphic to
\[
   S^1 \times D^3,
\]
and $M_2$ is homeomorphic to
\[
   S^2 \times D^2.
\]
$M_2$ can be covered by two subsets with disjoint interiors, each homeomorphic to
\[
   S^1 \times D^3.
\]
Together with $M_1$, these subsets cover $S^4$.
After smoothing corners where needed and taking an appropriate filtration of the sphere, we are in the situation of Theorem~\ref{thm:2.5}, whence the existence of a round function with three critical circles on the four-dimensional sphere follows.

The five-dimensional sphere is covered analogously by two subsets
\[
   M_1 = S^1 \times D^4\quad\text{and}\quad M_2 = S^3 \times D^2.
\]
The sphere $S^3$, and hence $M_2$, admits a covering by two subsets with disjoint interiors, each homeomorphic to
\[
   S^1 \times D^4.
\]
By Theorem~\ref{thm:2.5}, on the five-dimensional sphere there exists a round function with three critical circles.

It is not difficult to see that an analogous treatment of higher-dimensional spheres yields the statement of the theorem.
\qed

\nocite{*}
\printbibliography[heading=bibintoc,title={References}]

\end{document}